\RequirePackage{fix-cm}
\documentclass[smallextended]{svjour3}       
\smartqed  
\usepackage{amssymb}
\usepackage[leqno]{amsmath}
\usepackage{cases}
\usepackage{graphicx}
\usepackage{subfigure}
\usepackage{epstopdf}
\usepackage{epsfig}
\usepackage{float}
\usepackage{multirow}
\usepackage{float}
\usepackage{enumerate}
\usepackage{mathrsfs}
\usepackage{mathabx}
\usepackage{bm}
\usepackage{geometry}
\usepackage{tabularx}
\usepackage[numbers,sort&compress]{natbib}
\usepackage{color}
\usepackage{hyperref}
\usepackage[justification=centering]{caption}

\let \mr=\mathrm

\usepackage{epstopdf}
\begin{document}

\title{Semi-robust equal-order hybridized discontinuous Galerkin discretizations for the Oseen problem}

\titlerunning{Semi-robust equal-order HDG methods for the Oseen problem}        
 
\author{Xiaoqi Ma   \and
Jin Zhang }


\institute{Corresponding author:  Jin Zhang \at
              School of Mathematics and Statistics, Shandong Normal University,
Jinan 250014, China\\
              \email{jinzhangalex@sdnu.edu.cn}           
           \and
           Xiaoqi Ma  \at
           School of Mathematics and Statistics, Shandong Normal University,
Jinan 250014, China\\
            \email{xiaoqiMa@hotmail.com}  
}

\date{Received: date / Accepted: date}

\maketitle

\begin{abstract}
This paper introduces a unified analysis framework of equal-order hybridized discontinuous Galerkin (HDG) discretizations for the Oseen problem. The general framework covers standard HDG, embedded discontinuous Galerkin, and embedded-hybridized discontinuous Galerkin methods. While such equal-order schemes are computationally attractive, they generally violate the discrete inf-sup condition, leading to pressure instabilities. To address this issue, we introduce a symmetric pressure stabilization term. We establish the stability of the proposed methods and, by using carefully chosen projection operators, derive a semi-robust error estimate of order $k$ in an energy norm---where ``semi-robust'' means the error constant is independent of negative powers of the viscosity $\nu > 0$. Numerical experiments are provided to confirm the theoretical convergence rates and to illustrate the effectiveness of these methods.
\keywords{Oseen equation \and equal-order HDG methods \and quasi-uniform meshes \and convergence}
\subclass{76D07\and 65N12 \and 65N30}
\end{abstract}
%
%
%
\section{Introduction}
The numerical simulation of incompressible fluid flows, governed by the Navier--Stokes equations, is a cornerstone of computational fluid dynamics. In many practical scenarios, such as in iterative solution strategies or in time-dependent simulations, the problem is often linearized, leading to the Oseen equation. This model captures the essential challenges of incompressible flow, namely the coupling between velocity and pressure through the incompressibility constraint, and the interplay between viscous diffusion and convection \cite{Bur1Fer2:2006--C, Gar1Joh2:2021-motified}. The accurate and stable discretization of these equations, particularly in the convection-dominated regime where the viscosity is small, remains a topic of intense research.

Discontinuous Galerkin (DG) methods have gained prominence for their excellent stability properties, high-order accuracy, and suitability for complex geometries and adaptive meshing \cite{Coc1Kan2:2002--L, Coc1Kan2:2004--motified, Akb1Lin2:2018--motified}. Among the various DG variants, the hybridizable discontinuous Galerkin (HDG) method introduces a skeletal hybrid variable on element interfaces, enabling a substantial reduction in the globally coupled degrees of freedom through static condensation while preserving DG's advantageous properties \cite{Coc1Gop2:2009--motified, Kir1Rhe2:2019--A, Coc1Gop2Ngu3:2011--A, Rhe1Well2:2017--A, Coc1Say2:2014--D}. To further enhance efficiency, the embedded discontinuous Galerkin (EDG) method was developed, which enforces global continuity on these hybrid unknowns, leading to an even more compact global system \cite{Guz1Coc2Sto3:2007--motified, Lab1Well2:2007--motified, Lab1Well2:2012--E, Han1Hou2:2021--motified}. More recently, the embedded hybridized discontinuous Galerkin (E-HDG) framework \cite{Rhe1Well2:2020--motified} has been proposed to combine the computational efficiency of EDG with the pressure-robustness inherent to HDG. Collectively, these approaches represent a powerful and versatile framework for flow simulations, especially in convection-dominated regimes.

A central challenge for any mixed method in incompressible flow is satisfying the discrete inf-sup condition \cite{Bre1For2:1991-M}, which is essential for suppressing spurious pressure modes and ensuring stability. This requirement typically precludes the use of computationally attractive equal-order approximations (e.g., $P_{k}-P_{k}$ for $k\ge 1$) for velocity and pressure, as they are inherently inf-sup unstable. While inf-sup stable elements (e.g., $P_{k}-P_{k-1}$) are well-established, they complicate implementation and increase memory overhead. Consequently, various stabilization techniques have been developed for equal-order elements within the continuous finite element framework \cite{Hug1Fra2:1986-motified, Fra1Hug2:1993-C, Ahm1Cha2:2017--A, de1Gar2Joh3:2019--E, Gan1Mat2Tob3:2008--L, Bur1Fer2:2006--C, Bur1Fer2:2007--C, Bur1Han2:2006--E, Joh1Kno2:2020-F, Bur1Fer2:2008--G, Bos1Vol2:2021--S, Gan1Min2:2021--A, Coc1Kan2Sch3:2009--motified, Are1Kar2:2022--E}.

However, the analysis of stabilized equal-order formulations within a unified hybridized DG framework for the Oseen equations remains less mature. In 2017, S. Rhebergen et al. conducted a detailed stability and error analysis for equal-order HDG formulations for the Stokes equation \cite{Rhe1Well2:2017--A}. Subsequently, Y. Hou et al. analyzed equal-order HDG, EDG and E-HDG methods for the Stokes equation with a small pressure penalty parameter \cite{Hou1Han2:2021--motified}. These pioneering studies demonstrated the feasibility of stable equal-order HDG-type methods for incompressible Stokes flow. 
While there exist some analyses for the Oseen problem, they are predominantly focused on HDG methods within the local discontinuous Galerkin (LDG) framework \cite{Ces1Coc2--2013A, Gia1Sev2:2020--T, Sol1Var2:2022--motified2, Tu1Zha2:2023--B}. It is true that LDG-HDG methods allow for the local elimination of flux variables \cite{Cas1Seq2:2013--C}; however, interior penalty (IP) HDG methods provide a strictly primal alternative. IP-HDG methods inherently avoid auxiliary unknowns, resulting in a simpler and more compact implementation structure. Despite this, a critical gap persists regarding IP-HDG formulations for the Oseen problem, particularly in convection-dominated regimes. Most existing estimates depend on inverse powers of the viscosity $\nu$, failing to achieve semi-robustness---i.e., error bounds independent of $\nu$. This limitation significantly restricts the theoretical reliability of these methods for high-Reynolds-number flows.

This paper presents a unified analysis of equal-order HDG, EDG, and E-HDG methods for the incompressible Oseen problem. To circumvent the inf-sup condition, we introduce a symmetric pressure stabilization term. We then provide a stability analysis for the proposed schemes. The core of our theoretical effort lies in the derivation of semi-robust error estimates. By using carefully chosen projection operators, an optimal convergence rate of $\mathcal{O}(h^{k})$ can be derived in an energy norm, where the constant is independent of $\nu^{-1}$. This property is critical for ensuring the method's accuracy and reliability in the small-viscosity regime.

The outline of this article is as follows: Section 2 introduces the incompressible Oseen problem. In Section 3, we formulate the equal-order HDG, EDG, and E-HDG methods on quasi-uniform meshes. In Section 4, some interpolation operators and corresponding interpolation estimates are obtained. In Section 5, we establish error estimate with the constant independent of inverse powers of the viscosity. Finally, we present numerical experiments that validate our theoretical findings.

Let $\Omega \subset \mathbb{R}^{d}$ ($d= 2$ or $3$) be a bounded, convex polygonal ($d=2$) or polyhedral ($d=3$) domain. For any bounded domain $D \subseteq \Omega$, we consider the Sobolev spaces $W^{m, q}(D)$, associated with norms $||\cdot||_{m, q, D}$ for $m \ge 0$, and $q \ge 1$. When $m=0$, $W^{0, q}(D) =L^{q}(D)$ and when $q=2$, $ W^{m, 2}(D)= H^{m}(D)$. The norm and semi-norm of $H^{m}(D)$ are denoted by $||\cdot||_{m, D}$ and $|\cdot|_{m, D}$, respectively. For $m=0$ and $q=2$, we write $L^{2}(D):=H^{0}(D)$, whose inner product is denoted by $(\cdot, \cdot)_{D}$, with associated norm $||\cdot||_{D}$. When $D=\Omega$, we omit the subscript $D$ in both the inner product and the norm for simplicity.

\section{The Oseen problem}

The incompressible Oseen problem is formulated as follows:
\begin{equation}\label{eq:SSSS-1}
\left\{
\begin{aligned}
& -\nu\Delta u+(b\cdot\nabla)u+\sigma u+\nabla p=f&&\quad \text{in $\Omega$},\\
& \nabla\cdot u=0&&\quad \text{in $\Omega$},\\
&u=0&&\quad\text{on $\partial\Omega$},
\end{aligned}
\right.
\end{equation}
where $0<\nu<1$ is the kinematic viscosity, $\sigma>0$ is a reaction coefficient, and $b\in [W^{1, \infty}(\Omega)]^{d}$ is a given divergence-free velocity field (i.e., $\nabla\cdot  b=0$). Introduce the function spaces 
$$V := [H_{0}^{1}(\Omega)]^{d},\quad Q := L^{2}_{0}(\Omega)=\{q\in L^{2}(\Omega), \int_{\Omega} q \mr{d}x=0\}.$$ The weak formulation of \eqref{eq:S-1} reads as follows: Given $f\in [L^{2}(\Omega)]^{d}$, find $(u, p)\in V\times Q$ such that, for all $(v, q)\in V\times Q$,
\begin{equation*}\label{eq:S-2}
\left\{
\begin{aligned}
&\nu(\nabla u, \nabla v) + ((b\cdot \nabla)u, v) + (\sigma u, v)- (p, \nabla\cdot v) = F(v),\\
&(q, \nabla\cdot u) = 0
\end{aligned}
\right.
\end{equation*}
with $F(v)=\int_{\Omega}f\cdot v\mr{d}x$. The weak formulation is well posed by Babu$\check{s}$ka-Brezzi theory for all $\nu > 0$; see \cite{Gir1Rav2:2012--F}.

\section{The equal-order HDG methods}
\subsection{Mesh-related notation}
Let $\{\mathcal{T}_{h}\}_{0< h< 1}$ be a family of shape-regular and quasi-uniform triangulations of the domain $\Omega$. For each $\mathcal{T}_{h}$, we define the mesh size as $h = \max\limits_{K\in \mathcal{T}_{h}}h_{K}$, where $h_{K}$ denotes the diameter of element $K$. Let $\mathcal{F}_{h}$ be the set of all facets (edges for $d=2$, faces for $d=3$) in $\mathcal{T}_{h}$, and let $\mathcal{F} := \bigcup\limits_{F \in \mathcal{F}_h} F$ denote the mesh skeleton. We denote the boundary of an element $K$ by $\partial K$, and the outward unit normal vector on $\partial K$ by $n$. We introduce the trace operator $\zeta: H^{l}(\Omega)\rightarrow H^{l-1/2}(\mathcal{F})$ ($l\ge 1$) which restricts functions defined on $\Omega$ to the skeleton $\mathcal{F}$.

\subsection{Fnite element spaces and norms}
Denote the following discontinuous finite element spaces on $\Omega$ by
\begin{equation*}
\begin{aligned}
&V_{h}=\{v_{h}\in [L^{2}(\Omega)]^{d}: v_{h}\in [P_{k}(K)]^{d}, \forall K\in\mathcal{T}_{h}\},\\
&Q_{h}=\{q_{h}\in L^{2}(\Omega): q_{h}\in P_{k}(K), \forall K\in \mathcal{T}_{h}, \int_{\Omega} q_{h}dx=0\},
\end{aligned}
\end{equation*}
where $P_{k}(K)$ is the space of polynomials of degree at most $k$ on the element $K$. On the mesh skeleton $\mathcal{F}$, we define the corresponding discontinuous facet spaces:
\begin{equation*}
\begin{aligned}
&\bar{V}_{h}=\{\bar{v}_{h}\in [L^{2}(\mathcal{F})]^{d}: \bar{v}_{h}\in [P_{k}(F)]^{d},\forall F\in\mathcal{F}_{h}, \quad \bar{v}_{h}=0\quad \text{on $\partial\Omega$}\},\\
&\bar{Q}_{h}=\{\bar{q}_{h}\in L^{2}(\mathcal{F}): \bar{q}_{h}\in P_{k}(F), \forall F\in \mathcal{F}_{h}\},
\end{aligned}
\end{equation*}
where $P_{k}(F)$ denotes the space of polynomials of degree at most $k$ defined on the facet $F$. 

We define the extended function spaces $V_{h}^{*} = V_{h}\times \bar{V}_{h}$, $Q^{*}_{h}= Q_{h} \times \bar{Q}_{h}$ and $X_{h}^{*} = V_{h}^{*}\times Q_{h}^{*}$. Function pairs in $V_{h}^{*}$ and $Q^{*}_{h}$ will be denoted by boldface, e.g., ${\bm{v}}_{h} = (v_{h}, \bar{v}_{h})\in V_{h}^{*}$ and ${\bm{q}}_{h} = (q_{h}, \bar{q}_{h})\in Q^{*}_{h}$. We introduce the following function spaces:
\begin{equation*}
\begin{aligned}
&V(h)=V_{h}+[H_{0}^{1}(\Omega)]^{d}\cap [H^{2}(\Omega)]^{d},&&\quad Q(h)=Q_{h}+L_{0}^{2}(\Omega)\cap H^{1}(\Omega),\\
&\bar{V}(h)=\bar{V}_{h}+[H_{0}^{3/2}(\mathcal{F})]^{d},&&\quad \bar{Q}(h)=\bar{Q}_{h}+H_{0}^{1/2}(\mathcal{F}),
\end{aligned}
\end{equation*}
where $[H_{0}^{3/2}(\mathcal{F})]^{d}$ and $H_{0}^{1/2}(\mathcal{F})$ are the trace spaces of $[H_{0}^{1}(\Omega)]^{d}\cap [H^{2}(\Omega)]^{d}$ and $L_{0}^{2}(\Omega)\cap H^{1}(\Omega)$ on $\mathcal{F}$, respectively. The composite function spaces are defined as $V^{*}(h)=V(h)\times \bar{V}(h)$ and $Q^{*}(h)= Q(h)\times \bar{Q}(h)$.

For all $({\bm{v}}, {\bm{q}})\in V^{*}(h)\times Q^{*}(h)$, we define the natural norm derived from the bilinear form as
\begin{equation}\label{FF-3}
\begin{aligned}
&|||({\bm{v}}, {\bm{q}})|||_{\nu}^{2}=\nu|||{\bm{v}}|||^{2}+\sum_{K\in\mathcal{T}_{h}}\sigma||v||_{K}^{2}+|{\bm{v}}|_{up}^{2}+|{\bm{q}}|_{p}^{2},
\end{aligned}
\end{equation}
where 
\begin{equation*}
\begin{aligned}
&|||{\bm{v}}|||^{2}=\sum_{K\in \mathcal{T}_{h}}||\nabla v||_{K}^{2}+\sum_{K\in\mathcal{T}_{h}}\frac{\eta}{h_{K}}||\bar{v}-v||_{\partial K}^{2},\\
&|{\bm{v}}|_{up}^{2}=\frac{1}{2}\sum_{K\in\mathcal{T}_{h}}||| b\cdot n|^{\frac{1}{2}}(v-\bar{v})||_{\partial K}^{2},\\
&|{\bm{q}}|_{p}^{2}=\sum\limits_{K\in\mathcal{T}_{h}}\frac{\alpha}{\nu} h_{K}||\bar{q}-q||_{\partial K}^{2}.
\end{aligned}
\end{equation*}
The discrete trace inequalities \cite{Han1Hou2:2021--motified} will be used,
\begin{equation}
||v_{h}||_{\partial K} \le Ch_{K}^{-\frac{1}{2}}||v_{h}||_{K}\quad \forall v_{h}\in P_{k}(K), \forall K\in \mathcal{T}_{h}.\label{ZZ-2}
\end{equation}

\subsection{Weak formulation}
Now, the weak formulation of \eqref{eq:S-1} is that: Find $({\bm{u}}_{h}, {\bm{p}}_{h})\in X_{h}^{*}$ such that
\begin{equation}\label{eq:SD}
\left\{
\begin{aligned}
&a_{h}({\bm{u}}_{h}, {\bm{v}}_{h})+b_{h}({\bm{p}}_{h}, v_{h})+(\sigma u_{h}, v_{h})+o_{h}( b; {\bm{u}}_{h}, {\bm{v}}_{h})=(f, v_{h}),\\
&b_{h}({\bm{q}}_{h}, u_{h})-c_{h}({\bm{p}}_{h},{\bm{q}}_{h})=0,
\end{aligned}
\right.
\end{equation}
for all $({\bm{v}}_{h}, {\bm{q}}_{h})\in X_{h}^{*}$, where
\begin{equation*}
\begin{aligned}
&a_{h}({\bm{u}}_{h},{\bm{v}}_{h})=\sum_{K\in\mathcal{T}_{h}}\int_{K}\nu\nabla u_{h}:\nabla v_{h}\mr{d}x+\sum_{K\in\mathcal{T}_{h}}\int_{\partial K}\frac{\eta \nu}{h_{K}}(u_{h}-\bar{u}_{h})\cdot(v_{h}-\bar{v}_{h})\mr{d}s\\
&\quad\quad\quad\quad-\sum_{K\in\mathcal{T}_{h}}\int_{\partial K}[\nu(u_{h}-\bar{u}_{h})\cdot \frac{\partial v_{h}}{\partial n}+\nu\frac{\partial u_{h}}{\partial n}\cdot(v_{h}-\bar{v}_{h})]\mr{d}s,\\
&b_{h}({\bm{p}}_{h}, v_{h})=-\sum_{K\in\mathcal{T}_{h}}\int_{K}p_{h}\nabla\cdot v_{h}\mr{d}x+\sum_{K\in\mathcal{T}_{h}}\int_{\partial K}v_{h}\cdot n\bar{p}_{h}\mr{d}s,\\
&o_{h}( b; {\bm{u}}_{h}, {\bm{v}}_{h})=-\sum_{K\in\mathcal{T}_{h}}\int_{K}(u_{h}\otimes  b):\nabla v_{h}\mr{d}x+\sum_{K\in\mathcal{T}_{h}}\int_{\partial K}\frac{1}{2} b\cdot n(u_{h}+\bar{u}_{h})\cdot (v_{h}-\bar{v}_{h})\mr{d}s\\
&\quad\quad\quad\quad+\sum_{K\in\mathcal{T}_{h}}\int_{\partial K}\frac{1}{2}| b\cdot n|(u_{h}-\bar{u}_{h})\cdot (v_{h}-\bar{v}_{h})\mr{d}s,\\
&c_{h}({\bm{p}}_{h},{\bm{q}}_{h})=\sum_{K\in\mathcal{T}_{h}}\int_{\partial K}\frac{\alpha}{\nu} h_{K}(p_{h}-\bar{p}_{h})\cdot(q_{h}-\bar{q}_{h})\mr{d}s\\
&(f, v_{h})=\sum_{K\in\mathcal{T}_{h}}\int_{K}f\cdot v_{h}\mr{d}x.
\end{aligned}
\end{equation*}
Notice that $\eta>0$ is the velocity penalty parameter, and $\alpha>0$ is the pressure penalty parameter \cite{Rhe1Well2:2017--A}.

\begin{remark}
Following the unified framework established in \cite{Hou1Han2:2021--motified}, the distinction between the equal-order HDG, E-HDG, and EDG methods lies in the continuity constraints imposed on the trace spaces. Specifically, the global approximation space $X_{h}^{*}$ is defined as:
\begin{equation*}
\begin{aligned}
X_{h}^{*}&=V_{h}^{*}\times Q_{h}^{*}\\
&=\left\{
\begin{aligned}
&(V_{h}\times\bar{V}_{h})\times(Q_{h}\times \bar{Q}_{h}),&&\text{for HDG method},\\
&(V_{h}\times(\bar{V}_{h}\cap C^{0}(\mathcal{F})))\times(Q_{h}\times \bar{Q}_{h}),&&\text{for E-HDG method},\\
&(V_{h}\times (\bar{V}_{h}\cap C^{0}(\mathcal{F})))\times(Q_{h}\times (\bar{Q}_{h}\cap C^{0}(\mathcal{F}))),&&\text{for EDG method}.
\end{aligned}
\right.
\end{aligned}
\end{equation*}
\end{remark}

\begin{lemma}(Consistency) Let $(u, p) \in ([H_{0}^{1}(\Omega)]^{d} \cap [H^{2}(\Omega)]^{d}) \times (L^{2}_{0}(\Omega) \cap H^{1}(\Omega))$ and ${\bm{u}} = (u, \zeta(u))$ and ${\bm{p}} = (p, \zeta(p))$, where $(u, p)$ is the solution of \eqref{eq:SSSS-1}, 
\begin{equation*}
\left\{
\begin{aligned}
&a_{h}({\bm{u}}, {\bm{v}}_{h})+ b_{h}({\bm{p}}, v_{h})+ (\sigma u, v_{h})+ o_{h}( b; {\bm{u}}, {\bm{v}}_{h}) = ({\bm{f}}, {\bm{v}}_{h}),&&\quad \forall {\bm{v}}_{h}\in V_{h}^{*},\\
&b_{h}({\bm{q}}_{h},u) -c_{h}({\bm{p}}, {\bm{q}}_{h})= 0,&&\quad \forall {\bm{q}}_{h}\in Q_{h}^{*}.
\end{aligned}
\right.
\end{equation*}
\end{lemma}
\begin{proof}
This lemma follows directly by applying the theoretical framework developed in \citep[Lemma 4.1]{Rhe1Well2:2017--A}.
\end{proof}
\subsection{Stability of the method}
For the sake of the subsequent stability analysis, we first introduce the following two projections $\Pi_{K}$ and $\Pi_{F}$, where $\Pi_{K} : [H^{1}(\Omega)]^{d}\rightarrow V_{h}$ is $L^{2}$-projection such that
\begin{equation*}
\int_{K}(\Pi_{K}v - v) \cdot y_{h} \mr{d}x = 0\quad \forall y_{h}\in [P_{k}(K)]^{d}
\end{equation*}
for all $K\in \mathcal{T}_{h}$, and $\Pi_{F} : [H^{1}(\Omega)]^{d}\rightarrow \bar{V}_{h}$ is the $L^{2}$-projection,
\begin{equation*} 
\sum_{K\in\mathcal{T}_{h}}\int_{\partial K}(\Pi_{F}v - v)\cdot \bar{y}_{h}\mr{d}s = 0\quad \forall \bar{y}_{h}\in [P_{k}(F)]^{d} 
\end{equation*}
for all $F\in \mathcal{F}_{h}$. The following two estimates will be used,
\begin{align}
&||v - \Pi_{K}v||_{F} \le ch_{K}^{\frac{1}{2}}|v|_{1, K}\quad&& \forall F\in \mathcal{F}_{h}, F\subset \partial K, \forall K \in \mathcal{T}_{h},\label{eq:S-1}\\
&||\Pi_{K}v- \Pi_{F}v||^{2}_{\partial K} \le ch_{K}||v||^{2}_{1, K}\quad&& \forall K\in\mathcal{T}_{h},\label{eq:S-2}
\end{align}
where $c > 0$ is independent of $h_{K}$ and $\nu$. The first inequality is due to \citep[Lemma 1.59]{Di1Ern2:2012--M}, and the second is due to \citep[Proposition 3.9]{Coc1Gop2Ngu3:2011--A}.
\begin{lemma}\label{D-1}
\cite[Lemma 4.4]{Rhe1Well2:2017--A} For every ${\bm{p}}_{h}=(p_{h}, \bar{p}_{h})\in Q_{h}^{*}$, there exists a $w_{h}\in V_{h}$ such that
\begin{equation*}
 b_{h}({\bm{p}}_{h}, -w_{h}) \ge ||p_{h}||^{2} - c\nu^{\frac{1}{2}}\alpha^{-\frac{1}{2}} b_{c}^{-1}|{\bm{p}}_{h}|_{p} ||p_{h}||
 \end{equation*}
for the constant $c > 0$ independent of $\nu$ and $h_{K}$. Here $|{\bm{p}}_{h}|_{p}:=\left(\sum\limits_{K\in\mathcal{T}_{h}}\int_{\partial K}\frac{\alpha}{\nu}h_{K}(\bar{p}_{h}-p_{h})^{2}\mr{d}s\right)^{\frac{1}{2}}$.
\end{lemma}
\begin{proof}
For all $p\in L_{0}^{2}(\Omega)$ there exists a $v_{p}\in [H_{0}^{1}(\Omega)]^{d}$ such that
\begin{equation*}
p=\nabla\cdot v_{p},\quad  b_{c}||v_{p}||_{1}\le ||p||,
\end{equation*}
where $ b_{c} > 0$ is a constant depending only on $\Omega$ (see \citep[Theorem 6.5]{Di1Ern2:2012--M}). For a $p_{h}\in Q_{h}$, there exists a $v_{p_{h}}\in [H_{0}^{1}(\Omega)]^{d}$ such that 
\begin{equation}\label{SD-1}
\nabla\cdot v_{p_{h}}=p_{h},\quad  b_{c}||v_{p_{h}}||_{1}\le ||p_{h}||.
\end{equation}
It then follows that
\begin{equation*}
\begin{aligned}
||p_{h}||^{2}&=\int_{\Omega}(p_{h})^{2}\mr{d}x=\int_{\Omega}p_{h}\nabla\cdot v_{p_{h}}\mr{d}x\\
&=\sum_{K\in\mathcal{T}_{h}}\int_{\partial K}p_{h}v_{p_{h}}\cdot n\mr{d}s-\sum_{K\in\mathcal{T}_{h}}\int_{K}\nabla p_{h}\cdot v_{p_{h}}\mr{d}x.
\end{aligned}
\end{equation*}
According to the projection $\Pi_{K}$, one has $\int_{K}\nabla p_{h}\cdot (\Pi_{K}v_{p_{h}}-v_{p_{h}})\mr{d}x=0$,
\begin{equation}\label{FF-1}
||p_{h}||^{2}=\sum_{K\in\mathcal{T}_{h}}\int_{\partial K}p_{h}v_{p_{h}}\cdot n\mr{d}s-\sum_{K\in\mathcal{T}_{h}}\int_{K}\nabla p_{h}\cdot (\Pi_{K}v_{p_{h}})\mr{d}x.
\end{equation}
Using the definition of $b_{h}$ and take $w_{h}=\Pi_{K}v_{p_{h}}$, we have
\begin{equation}\label{FF-2}
\begin{aligned}
&b_{h}({\bm{p}}_{h}, \Pi_{K}v_{p_{h}})\\
&=-\sum_{K\in\mathcal{T}_{h}}\int_{K}p_{h}\nabla\cdot (\Pi_{K}v_{p_{h}})\mr{d}x+\sum_{K\in\mathcal{T}_{h}}\int_{\partial K}\Pi_{K}v_{p_{h}}\cdot n\bar{p}_{h}\mr{d}s\\
&=\sum_{K\in\mathcal{T}_{h}}\int_{K}\nabla p_{h}\cdot (\Pi_{K}v_{p_{h}})\mr{d}x+\sum_{K\in\mathcal{T}_{h}}\int_{\partial K}(\bar{p}_{h}-p_{h})\Pi_{K}v_{p_{h}}\cdot n\mr{d}s.
\end{aligned}
\end{equation}
Then from \eqref{FF-1} and \eqref{FF-2}, we can establish the relationship between $b_{h}(\cdot, \cdot)$ and $||p_{h}||$,
\begin{equation*}
\begin{aligned}
||p_{h}||^{2}&=-b_{h}({\bm{p}}_{h}, \Pi_{K}v_{p_{h}})+\sum_{K\in\mathcal{T}_{h}}\int_{\partial K}(\bar{p}_{h}-p_{h})\Pi_{K}v_{p_{h}}\cdot n\mr{d}s\\
&\quad-\sum_{K\in\mathcal{T}_{h}}\int_{\partial K}(\bar{p}_{h}-p_{h})v_{p_{h}}\cdot n\mr{d}s\\
&=-b_{h}({\bm{p}}_{h}, \Pi_{K}v_{p_{h}})+\sum_{K\in\mathcal{T}_{h}}\int_{\partial K}(\bar{p}_{h}-p_{h})(\Pi_{K}v_{p_{h}}-v_{p_{h}})\cdot n\mr{d}s.
\end{aligned}
\end{equation*}
It is easy to derive that
\begin{equation*}
-b_{h}({\bm{p}}_{h}, \Pi_{K}v_{p_{h}})=||p_{h}||^{2}-\sum_{K\in\mathcal{T}_{h}}\int_{\partial K}(\bar{p}_{h}-p_{h})(\Pi_{K}v_{p_{h}}-v_{p_{h}})\cdot n\mr{d}s,
\end{equation*}
where according to the approximate property of $\Pi_{K}$ \eqref{eq:S-2}, i.e.,
\begin{equation*}
||\Pi_{K}v_{p_{h}}-v_{p_{h}}||^{2}_{\partial K}\le ch_{K}|v_{p_{h}}|_{1, K}^{2},
\end{equation*}
one has
\begin{equation*}
\begin{aligned}
&|\sum_{K\in\mathcal{T}_{h}}\int_{\partial K}(\bar{p}_{h}-p_{h})(\Pi_{K}v_{p_{h}}-v_{p_{h}})\cdot n\mr{d}s|\\
&\le \left(\sum_{K\in\mathcal{T}_{h}}\int_{\partial K}h_{K}(\bar{p}_{h}-p_{h})^{2}\mr{d}s\right)^{\frac{1}{2}}\left(\sum_{K\in\mathcal{T}_{h}}\int_{\partial K}\frac{1}{h_{K}}(\Pi_{K}v_{p_{h}}-v_{p_{h}})^{2}\mr{d}s\right)^{\frac{1}{2}}\\
&\le c\left(\sum_{K\in\mathcal{T}_{h}}\int_{\partial K}h_{K}(\bar{p}_{h}-p_{h})^{2}\mr{d}s\right)^{\frac{1}{2}}||v_{p_{h}}||\\
&\le c\left(\frac{\nu}{\alpha}\right)^{\frac{1}{2}}\left(\sum_{K\in\mathcal{T}_{h}}\int_{\partial K}\frac{\alpha}{\nu}h_{K}(\bar{p}_{h}-p_{h})^{2}\mr{d}s\right)^{\frac{1}{2}}||v_{p_{h}}||_{1}\\
&\le c b_{c}^{-1}\nu^{\frac{1}{2}}\alpha^{-\frac{1}{2}}|{\bm{p}}_{h}|_{p}||p_{h}||.
\end{aligned}
\end{equation*}
This completes the proof.
\end{proof}

Next we introduce the global form
$$B_{h}(({\bm{u}}, {\bm{p}}); ({\bm{v}}, {\bm{q}}))=a_{h}({\bm{u}}, {\bm{v}})+b_{h}({\bm{p}}, v)+(\sigma u, v)+o_{h}( b; {\bm{u}}, {\bm{v}})-b_{h}({\bm{q}}, u)+c_{h}({\bm{p}}, {\bm{q}}).$$
Due to the absence of the $L^2$-norm of the pressure in the natural norm \eqref{FF-3}, the well-posedness of the discrete problem \eqref{eq:SD} cannot be guaranteed. To address this issue, we define the following norm, 
\begin{equation}\label{GG-1}
 |||({\bm{v}}_{h}, {\bm{q}}_{h})|||^{2}:= |||({\bm{v}}_{h}, {\bm{q}}_{h})|||_{\nu}^{2}+\gamma||q_{h}||^{2}\quad\forall ({\bm{v}}_{h}, {\bm{q}}_{h})\in X_{h}^{*}.
\end{equation}
\begin{lemma}\label{WW-2}
(Coercivity) There exists a constant $C_{s} > 0$ independent of $h_{K}$ and $\nu$, such that 
\begin{equation*}
 \inf_{({\bm{u}}_{h}, {\bm{p}}_{h})\in X_{h}^{*}}\sup_{({\bm{v}}_{h}, {\bm{q}}_{h})\in X_{h}^{*}}\frac{B_{h}(({\bm{u}}_{h}, {\bm{p}}_{h}); ({\bm{v}}_{h}, {\bm{q}}_{h}))}{|||({\bm{u}}_{h}, {\bm{p}}_{h})|||  |||({\bm{v}}_{h}, {\bm{q}}_{h})|||}\ge C_{s}.
\end{equation*}
\end{lemma}
\begin{proof}
Let $({\bm{v}}_{h}, {\bm{q}}_{h})=({\bm{u}}_{h}, {\bm{p}}_{h})+\delta(-{\bm{w}}_{h}, {\bm{0}})$, where $\delta>0$ is a reasonably selected parameter and take ${\bm{w}}_{h}=(\Pi_{K}v_{p_{h}}, \Pi_{F}v_{p_{h}})$. 

Firstly, it is easy to obtain that $$B_{h}(({\bm{u}}_{h}, {\bm{p}}_{h}); ({\bm{u}}_{h}, {\bm{p}}_{h}))\ge c_{1}|||({\bm{u}}_{h}, {\bm{p}}_{h})|||_{\nu}^{2}.$$


Then according to Lemma \ref{D-1} and \eqref{SD-1}, one derives
\begin{equation*}
\begin{aligned}
&B_{h}(({\bm{u}}_{h}, {\bm{p}}_{h}); (-{\bm{w}}_{h}, {\bm{0}}))\\
&=-a_{h}({\bm{u}}_{h}, {\bm{w}}_{h})-b_{h}({\bm{p}}_{h}, w_{h})-(\sigma u_{h}, w_{h})-o_{h}( b; {\bm{u}}_{h}, {\bm{w}}_{h})\\
&\ge -b_{h}({\bm{p}}_{h}, w_{h})-|a_{h}({\bm{u}}_{h}, {\bm{w}}_{h})|-|(\sigma u_{h}, w_{h})|-|o_{h}( b; {\bm{u}}_{h}, {\bm{w}}_{h})|\\
&\ge ||p_{h}||^{2}-c b_{c}^{-1}\nu^{\frac{1}{2}}\alpha^{-\frac{1}{2}}||p_{h}|| |{\bm{p}}_{h}|_{p}-|a_{h}({\bm{u}}_{h}, {\bm{v}}_{h})|-|(\sigma u_{h}, v_{h})|-|o_{h}( b; {\bm{u}}_{h}, {\bm{v}}_{h})|.
\end{aligned}
\end{equation*}
By using the theories of the $L^{2}$-projection (see \citep[Lemma 4]{Rhe1Well2:2017--A} and \citep[Lemma 6.11]{Di1Ern2:2012--M}) and \eqref{SD-1}, 
\begin{equation}\label{SD-9}
\begin{aligned}
\sum_{K\in\mathcal{T}_{h}}||\nabla(\Pi_{K}v_{p_{h}})||_{K}^{2}&\le c||v_{p_{h}}||_{1}^{2},
\end{aligned}
\end{equation}
 and from \eqref{eq:S-2}, we have
\begin{equation}\label{SD-8}
\sum_{K\in\mathcal{T}_{h}}\frac{1}{h_{K}}||\Pi_{F}v_{p_{h}}-\Pi_{K}v_{p_{h}}||_{\partial K}^{2}\le c||v_{p_{h}}||_{1}^{2}.
\end{equation}
Then apply the H$\ddot{o}$lder inequality and the discrete trace inequality \eqref{ZZ-2}, 
\begin{equation*}
|a_{h}({\bm{u}}_{h}, {\bm{w}}_{h})|\le c\nu^{\frac{1}{2}} b_{c}^{-1}(1+\eta^{\frac{1}{2}})(1+\eta^{-\frac{1}{2}})|||({\bm{u}}_{h}, {\bm{0}})|||_{\nu}||p_{h}||.
\end{equation*}
From the stability of $L^2$-projection and \eqref{SD-1},
\begin{equation*}
\begin{aligned}
|(\sigma u_{h}, w_{h})|&\le \sigma\left(\sum_{K\in\mathcal{T}_{h}}||u_{h}||_{K}^{2}\right)^{\frac{1}{2}}\left(\sum_{K\in\mathcal{T}_{h}}||w_{h}||_{K}^{2}\right)^{\frac{1}{2}}\\
&= \sigma\left(\sum_{K\in\mathcal{T}_{h}}||u_{h}||_{K}^{2}\right)^{\frac{1}{2}}\left(\sum_{K\in\mathcal{T}_{h}}||\Pi_{K}v_{p_{h}}||_{K}^{2}\right)^{\frac{1}{2}}\\
&\le \sigma\left(\sum_{K\in\mathcal{T}_{h}}||u_{h}||_{K}^{2}\right)^{\frac{1}{2}}\left(\sum_{K\in\mathcal{T}_{h}}||v_{p_{h}}||_{K}^{2}\right)^{\frac{1}{2}}\\
&\le c\sigma^{\frac{1}{2}}|||({\bm{u}}_{h}, {\bm{0}})|||_{\nu} ||v_{p_{h}}||_{1}\\
&\le c\sigma^{\frac{1}{2}} b_{c}^{-1}|||({\bm{u}}_{h}, {\bm{0}})|||_{\nu}||p_{h}||.
\end{aligned}
\end{equation*}
To manipulate the convection term, we apply the algebraic identity $\frac{1}{2}(A+B) = \frac{1}{2}(A-B) - (A-B) + A$. Expanding the terms yields:
\begin{equation*}
\begin{aligned}
&o_{h}( b; {\bm{u}}_{h}, {\bm{w}}_{h})\\&=-\sum_{K\in\mathcal{T}_{h}}\int_{K}(u_{h}\otimes  b):\nabla w_{h}\mr{d}x+\sum_{K\in\mathcal{T}_{h}}\int_{\partial K}\frac{1}{2}( b\cdot n)(u_{h}+\bar{u}_{h})\cdot (w_{h}-\bar{w}_{h})\mr{d}s,\\
&+\sum_{K\in\mathcal{T}_{h}}\int_{\partial K}\frac{1}{2}| b\cdot n|(u_{h}-\bar{u}_{h})\cdot (w_{h}-\bar{w}_{h})\mr{d}s\\
&=-\sum_{K\in\mathcal{T}_{h}}\int_{K}(u_{h}\otimes  b):\nabla w_{h}\mr{d}x+\sum_{K\in\mathcal{T}_{h}}\int_{\partial K}\frac{1}{2}( b\cdot n)(u_{h}-\bar{u}_{h})\cdot (w_{h}-\bar{w}_{h})\mr{d}s\\
&+\sum_{K\in\mathcal{T}_{h}}\int_{\partial K}( b\cdot n)(\bar{u}_{h}-u_{h})\cdot (w_{h}-\bar{w}_{h})\mr{d}s+\sum_{K\in\mathcal{T}_{h}}\int_{\partial K}( b\cdot n)u_{h}\cdot (w_{h}-\bar{w}_{h})\mr{d}s\\
&+\sum_{K\in\mathcal{T}_{h}}\int_{\partial K}\frac{1}{2}| b\cdot n|(u_{h}-\bar{u}_{h})\cdot (w_{h}-\bar{w}_{h})\mr{d}s.
\end{aligned}
\end{equation*}
Then from the H$\ddot{o}$lder inequality, \eqref{SD-9} and \eqref{SD-8}, we have
\begin{equation*}
\begin{aligned}
&|-\sum_{K\in\mathcal{T}_{h}}\int_{K}(u_{h}\otimes  b):\nabla w_{h}\mr{d}x|\\
&\le || b||_{\infty}\left(\sum_{K\in\mathcal{T}_{h}}||u_{h}||_{K}^{2}\right)^{\frac{1}{2}}\left(\sum_{K\in\mathcal{T}_{h}}||\nabla w_{h}||_{K}^{2}\right)^{\frac{1}{2}}\\
&\le C|| b||_{\infty}\left(\sum_{K\in\mathcal{T}_{h}}||u_{h}||_{K}^{2}\right)^{\frac{1}{2}}||v_{p_{h}}||_{1}\\
&\le C|| b||_{\infty} b_{c}^{-1}\sigma^{-\frac{1}{2}}|||({\bm{u}}_{h}, {\bm{0}})|||_{\nu}||p_{h}||,
\end{aligned}
\end{equation*} 
\begin{equation*}
\begin{aligned}
&|\sum_{K\in\mathcal{T}_{h}}\int_{\partial K}\frac{1}{2}( b\cdot n)(u_{h}-\bar{u}_{h})\cdot (w_{h}-\bar{w}_{h})\mr{d}s|\\
&\le Ch_{K}^{\frac{1}{2}}|| b||_{\infty}^{\frac{1}{2}}\left(\sum_{K\in\mathcal{T}_{h}}\int_{\partial K}| b\cdot n|(u_{h}-{\bar{u}}_{h})^{2}\mr{d}s\right)^{\frac{1}{2}}\left(\sum_{K\in\mathcal{T}_{h}}\int_{\partial K}\frac{1}{h_{K}}(w_{h}-{\bar{w}}_{h})^{2}\mr{d}s\right)^{\frac{1}{2}}\\
&\le Ch_{K}^{\frac{1}{2}}|| b||_{\infty}^{\frac{1}{2}}|||({\bm{u}}_{h}, {\bm{0}})|||_{\nu}\left(\sum_{K\in\mathcal{T}_{h}}\frac{1}{h_{K}}\Vert w_{h}-\bar{w}_{h}\Vert_{\partial K}^{2}\right)^{\frac{1}{2}}\\
&\le Ch_{K}^{\frac{1}{2}}|| b||_{\infty}^{\frac{1}{2}}|||({\bm{u}}_{h}, {\bm{0}})|||_{\nu}||v_{p_{h}}||_{1}\\
&\le Ch_{K}^{\frac{1}{2}}|| b||_{\infty}^{\frac{1}{2}} b_{c}^{-1}|||({\bm{u}}_{h}, {\bm{0}})|||_{\nu}||p_{h}||,
\end{aligned}
\end{equation*}
and
\begin{equation*}
\begin{aligned}
&\sum_{K\in\mathcal{T}_{h}}\int_{\partial K}\frac{1}{2}( b\cdot n)u_{h}\cdot (w_{h}-\bar{w}_{h})\mr{d}s\\
&\le C|| b||_{\infty}\left(\sum_{K\in\mathcal{T}_{h}}h_{K}||u_{h}||_{\partial K}^{2}\right)^{\frac{1}{2}}\left(\sum_{K\in\mathcal{T}_{h}}\frac{1}{h_{K}}||w_{h}-\bar{w}_{h}||_{\partial K}^{2}\right)^{\frac{1}{2}}\\
&\le C|| b||_{\infty}\left(\sum_{K\in\mathcal{T}_{h}}||u_{h}||_{K}^{2}\right)^{\frac{1}{2}}||v_{p_{h}}||_{1}\\
&\le C|| b||_{\infty}\sigma^{-\frac{1}{2}} b_{c}^{-1}|||({\bm{u}}_{h}, {\bm{0}})|||_{\nu} ||p_{h}||.
\end{aligned}
\end{equation*}
In summary, we obtain
\begin{equation*}
\begin{aligned}
&B_{h}(({\bm{u}}_{h}, {\bm{p}}_{h}); (-{\bm{w}}_{h}, {\bm{0}}))\\
&\ge ||p_{h}||^{2}-c b_{c}^{-1}\nu^{\frac{1}{2}}\alpha^{-\frac{1}{2}}||p_{h}|||{\bm{p}}_{h}|_{p}-c\nu^{\frac{1}{2}}(1+\eta^{\frac{1}{2}})(1+\eta^{-\frac{1}{2}}) b_{c}^{-1}|||({\bm{u}}_{h}, {\bm{p}}_{h})|||_{\nu} ||p_{h}||\\
&-C(\sigma^{\frac{1}{2}}+||b||_{\infty}\sigma^{-\frac{1}{2}}+||b||_{\infty}^{\frac{1}{2}}h^{\frac{1}{2}}) b_{c}^{-1}|||({\bm{u}}_{h}, {\bm{p}}_{h})|||_{\nu} ||p_{h}||\\
&\ge ||p_{h}||^{2}-c b_{c}^{-1}\nu^{\frac{1}{2}}\alpha^{-\frac{1}{2}}||p_{h}|||{\bm{p}}_{h}|_{p}-c_{2}(\nu^{\frac{1}{2}}+h^{\frac{1}{2}}+1)|||({\bm{u}}_{h}, {\bm{p}}_{h})|||_{\nu} ||p_{h}||.
\end{aligned}
\end{equation*}

Now we use $({\bm{v}}_{h}, {\bm{q}}_{h})=({\bm{u}}_{h}, {\bm{p}}_{h})+\delta(-{\bm{w}}_{h}, {\bm{0}})$, which yields
\begin{equation*}
\begin{aligned}
B_{h}(({\bm{u}}_{h}, {\bm{p}}_{h}), ({\bm{v}}_{h}, {\bm{q}}_{h}))&=B_{h}(({\bm{u}}_{h}, {\bm{p}}_{h}), ({\bm{u}}_{h}, {\bm{p}}_{h}))+\delta B_{h}(({\bm{u}}_{h}, {\bm{p}}_{h}), (-{\bm{w}}_{h}, {\bm{0}}))\\
&\ge c_{1}|||({\bm{u}}_{h}, {\bm{p}}_{h})|||_{\nu}^{2}+\delta ||p_{h}||^{2}-c\delta b_{c}^{-1}\nu^{\frac{1}{2}}\alpha^{-\frac{1}{2}}||p_{h}|||{\bm{p}}_{h}|_{p}\\
&-\delta c_{2}(\nu^{\frac{1}{2}}+h^{\frac{1}{2}}+1)|||({\bm{u}}_{h}, {\bm{p}}_{h})|||_{\nu}||p_{h}||\\
&\ge (c_{1}-\frac{\delta c_{2}\varepsilon_{2}}{2})|||({\bm{u}}_{h}, {\bm{p}}_{h})|||_{\nu}^{2}+(c_{1}-\frac{\delta c\varepsilon_{1}}{2})|{\bm{p}}_{h}|_{p}^{2}\\
&+\delta\left(1-\frac{c\alpha^{-1} b_{c}^{-2}\nu}{2\varepsilon_{1}}-\frac{c_{2}}{2\varepsilon_{2}}(\nu^{\frac{1}{2}}+h^{\frac{1}{2}}+1)^{2}\right)||p_{h}||^{2}\\
&\ge (c_{1}-\frac{\delta c_{2}\varepsilon_{2}}{2})|||({\bm{u}}_{h}, {\bm{p}}_{h})|||_{\nu}^{2}+(c_{1}-\frac{\delta c\varepsilon_{1}}{2})|{\bm{p}}_{h}|_{p}^{2}\\
&+\delta\left(1-\frac{c\alpha^{-1} b_{c}^{-2}}{2\varepsilon_{1}}-\frac{c_{2}}{2\varepsilon_{2}}\right)||p_{h}||^{2},
\end{aligned}
\end{equation*}
where $\varepsilon_{1}> 0$ and $\varepsilon_{2} > 0$ are free parameters that can be appropriately chosen. By selecting $\varepsilon_{1}$ and $\varepsilon_{2}$ large enough such that 
\begin{equation*}
\begin{aligned}
B_{h}(({\bm{u}}_{h}, {\bm{p}}_{h}), ({\bm{v}}_{h}, {\bm{q}}_{h}))\ge (c_{1}-\frac{\delta c_{2}\varepsilon_{2}}{2})|||({\bm{u}}_{h}, {\bm{p}}_{h})|||_{\nu}^{2}+(c_{1}-\frac{\delta c\varepsilon_{1}}{2})|{\bm{p}}_{h}|_{p}^{2}+\delta||p_{h}||^{2},
\end{aligned}
\end{equation*}
where $c_ {1}>0$ and $c_ {2}>0 $ are constants independent of $\nu$ and $h$.
Finally, choosing $\delta$ sufficiently small ensures the existence of a constant $c_{3} > 0$ such that
\begin{equation*}
\begin{aligned}
B_{h}(({\bm{u}}_{h}, {\bm{p}}_{h}), ({\bm{v}}_{h}, {\bm{q}}_{h}))&\ge c_{3}|||({\bm{u}}_{h}, {\bm{p}}_{h})|||^{2}.
\end{aligned}
\end{equation*}

Finally, according to the energy norm \eqref{FF-3}, we obtain
\begin{equation*}
\begin{aligned}
&|||({\bm{v}}_{h}, {\bm{q}}_{h})|||\\
&\le |||({\bm{u}}_{h}, {\bm{p}}_{h})|||+\delta|||(-{\bm{w}}_{h}, {\bm{0}})|||\\
&\le |||({\bm{u}}_{h}, {\bm{p}}_{h})|||+\delta\nu^{\frac{1}{2}}|||{\bm{w}}_{h}|||+\delta\sigma^{\frac{1}{2}}||w_{h}||+\delta\left(\frac{1}{2}\sum_{K\in\mathcal{T}_{h}}|| |b\cdot n|^{\frac{1}{2}}(w_{h}-\bar{w}_{h})||_{\partial K}^{2}\right)^{\frac{1}{2}}\\
&\le |||({\bm{u}}_{h}, {\bm{p}}_{h})|||+c\delta(1+\eta^{\frac{1}{2}})\nu^{\frac{1}{2}}||v_{p_{h}}||_{1}+\delta\sigma^{\frac{1}{2}}||v_{p_{h}}||_{1}+c\delta h^{\frac{1}{2}}||b||_{\infty}^{\frac{1}{2}}||v_{p_{h}}||_{1}\\
&\le |||({\bm{u}}_{h}, {\bm{p}}_{h})|||+c\delta b_{c}^{-1}\nu^{\frac{1}{2}}||p_{h}||+\delta\sigma^{\frac{1}{2}} b_{c}^{-1}||p_{h}||+\delta h^{\frac{1}{2}}||b||_{\infty}^{\frac{1}{2}} b_{c}^{-1}||p_{h}||\\
&\le |||({\bm{u}}_{h}, {\bm{p}}_{h})|||+c_{3}\delta(\nu^{\frac{1}{2}}+h^{\frac{1}{2}}+1) b_{c}^{-1}||p_{h}||\\
&\le |||({\bm{u}}_{h}, {\bm{p}}_{h})|||+c_{3}\delta b_{c}^{-1}||p_{h}||\\
&\le c_{4}|||({\bm{u}}_{h}, {\bm{p}}_{h})|||,
\end{aligned}
\end{equation*}
where $c_{3}, c_{4} > 0$ are constants independent of $h$ and $\nu$, and the last inequality follows from the fact that $||p_{h}||$ is controlled by the norm $|||({\bm{u}}_{h}, {\bm{p}}_{h})|||$.

Combining this bound with the previously established coercivity result, we obtain $B_{h}(({\bm{u}}_{h}, {\bm{p}}_{h}), ({\bm{v}}_{h}, {\bm{q}}_{h}))\ge c_{3}c_{4}^{-1}|||({\bm{u}}_{h}, {\bm{p}}_{h})||| |||({\bm{v}}_{h}, {\bm{q}}_{h})|||$.
\end{proof}


\section{Interpolation and interpolation estimates}
\subsection{Interpolation operators and approximation properties}
Next we introduce the following approximation and discretization errors for the velocity and the pressure, respectively: 
\begin{equation*}
\begin{aligned}
&{\bm{e}}_{u}=(e_{u}, e_{\hat{u}}), {\bm{\eta}}_{u}=(\eta_{u}, \eta_{\hat{u}}), {\bm{\xi}}_{u}=(\xi_{u}, \xi_{\hat{u}}),\\
&{\bm{e}}_{p}=(e_{p}, e_{\hat{p}}), {\bm{\eta}}_{p}=(\eta_{p}, \eta_{\hat{p}}), {\bm{\xi}}_{p}=(\xi_{p}, \xi_{\hat{p}}),\\
&{\bm{e}}_{u}= {\bm{u}}-{\bm{u}}_{h}={\bm{\eta}}_{u}+{\bm{\xi}}_{u},\\
&{\bm{e}}_{p} = {\bm{p}}-{\bm{p}}_{h} = {\bm{\eta}}_{p} + {\bm{\xi}}_{p} ,\\
&{\bm{\eta}}_{u}={\bm{u}}-\Pi_{U} {\bm{u}},\quad {\bm{\xi}}_{u}=\Pi_{U}{\bm{u}}-{\bm{u}}_{h},\\
&{\bm{\eta}}_{p}={\bm{p}} - \Pi_{Q}{\bm{p}},\quad {\bm{\xi}}_{p}=\Pi_{Q}{\bm{p}}-{\bm{p}}_{h},
\end{aligned}
\end{equation*}
where $\Pi_{U}{\bm{u}}=(\Pi_{K}u, \Pi_{F}u)$ denote the $L^{2}$-projection onto $V_{h}\times\bar{V}_{h}$, $\Pi_{Q}{\bm{p}}=({\mathcal{I}}_{K}p, \mathcal{I}_{F}p)$ are Lagrange interpolations of order $k$ on $Q_{h}\times\bar{Q}_{h}$, respectively. 

On the one hand, assuming $\Pi_{K}$ is the local $L^{2}$-projection operator, we have the following approximation property (see \cite{Cro18jTho2:1987--motified, Han1Hou2:2021--motified, Di1Ern2:2012--M}):
\begin{equation}\label{P-111}
|| v- \Pi_{K} v||_{K} + h_{K}^{\frac{1}{2}}\Vert v-\Pi_{K} v\Vert_{\partial K} \le Ch_{K}^{k+1}|v|_{H^{k+1}(K)}
\end{equation}
for all $v\in H^{k+1}(K)$. Similarly, according to the projection results in \citep[Lemma 5.1]{Che1:2021--O}, we define $\Pi_{F}z\in P_{k}(F)$ for any $z\in L^{2}(F)$ by the orthogonality condition $(\Pi_{F}z, v)_{F}=(z, v)_{F}$ for all $v\in P_{k}(F)$. For any function $z\in W^{k+1, m}(F)$, the following estimate holds:
\begin{align}
|| z-\Pi_{F}z ||_{L^{m}(F)}\le C h_{F}^{k+1} |z|_{W^{k+1, m}(F)}, \quad m=2, \infty, \label{eq:interpolation-theory-3}
\end{align}
where $h_{F}$ denotes the length of $F\in\mathcal{F}_{h}$.

On the other hand, assuming $\mathcal{I}_{K}$ is the Lagrange interpolation, we have the following approximation property (see \cite{Cro18jTho2:1987--motified, Han1Hou2:2021--motified, Di1Ern2:2012--M}):
\begin{equation}\label{BB-1}
|u -\mathcal{I}_{K}u|_{W^{j, p}(K)}\le Ch^{s-j}|u|_{W^{s, p}(K)},\quad 0 \le j \le s \le k+1, 
\end{equation}
where $s > d/p$ when $1 < p \le\infty$ and $s \ge d$ when $p = 1$.

\subsection{Interpolation error estimate}
\begin{theorem}\label{DDD-4}
Let $(u, p)\in [H^{k+1}(\Omega)]^{d}\times H^{k+1}(\Omega)$ with $k \ge 1$. Then, there exists a constant $C > 0$ independent of $h$ and $\nu$, such that
\begin{equation*}
\begin{aligned}
|||({\bm{\eta}}_{u}, {\bm{\eta}}_{p})|||&\le C\left(\nu^{\frac{1}{2}}h^{k}+|| b||_{\infty}^{\frac{1}{2}}h^{k+\frac{1}{2}}+\sigma^{\frac{1}{2}} h^{k+1}\right)||u||_{k+1}+Ch^{k+1}||p||_{k+1}. 
\end{aligned}
\end{equation*}
Note that $C$ is dependent on the parameters $\eta$, $\alpha$ and $\gamma$.
\end{theorem}
\begin{proof}
By the definition of the energy norm in \eqref{GG-1}, we have
\begin{equation*}
\begin{aligned}
|||({\bm{\eta}}_{u}, {\bm{\eta}}_{p})|||^{2}&=\nu \sum_{K\in\mathcal{T}_{h}}||\nabla\eta_{u}||_{K}^{2}+\sum_{K\in\mathcal{T}_{h}}\frac{\nu\eta}{h_{K}}||\eta_{\hat{u}}-\eta_{u}||_{\partial K}^{2}\\
&+\sigma\sum_{K\in\mathcal{T}_{h}}||\eta_{u}||^{2}_{K}+\frac{1}{2}\sum_{K\in\mathcal{T}_{h}}||| b\cdot n|^{\frac{1}{2}}(\eta_{u}-\eta_{\hat{u}})||_{\partial K}^{2}\\
&+\sum_{K\in\mathcal{T}_{h}}\left\Vert\left(\frac{\alpha}{\nu}\right)^{\frac{1}{2}}(\eta_{p}-\eta_{\hat{p}})\right\Vert_{\partial K}^{2}+\gamma||\eta_{p}||^{2}.
\end{aligned}
\end{equation*}
From \eqref{P-111}, one has
\begin{equation*}
\begin{aligned}
\nu\sum_{K\in\mathcal{T}_{h}}||\nabla\eta_{u}||_{K}^{2}&\le C\nu h^{2k}||u||_{k+1}^{2}
\end{aligned}
\end{equation*}
Using the triangle inequality, along with \eqref{P-111} and \eqref{eq:interpolation-theory-3}, we have
\begin{equation*}
\begin{aligned}
\sum_{K\in\mathcal{T}_{h}}\frac{\nu\eta}{h_{K}}||\eta_{\hat{u}}-\eta_{u}||_{\partial K}^{2}\le C\sum_{K\in\mathcal{T}_{h}}\frac{\nu}{h_{K}}(||\eta_{\hat{u}}||_{\partial K}^{2}+||\eta_{u}||_{\partial K}^{2})\le C\nu h^{2k}||u||_{k+1}^{2}.
\end{aligned}
\end{equation*}
Applying \eqref{P-111}, it follows directly that
\begin{equation*}
\sigma\sum_{K\in\mathcal{T}_{h}}||\eta_{u}||^{2}_{K}\le C\sigma h^{2(k+1)}||u||_{k+1}^{2}.
\end{equation*}
Using the interpolation error estimate \eqref{P-111} and \eqref{eq:interpolation-theory-3}, we get
\begin{equation*}
\frac{1}{2}\sum_{K\in\mathcal{T}_{h}}||| b\cdot n|^{\frac{1}{2}}(\eta_{u}-\eta_{\hat{u}})||_{\partial K}^{2}\le C|| b||_{\infty}h^{2k+1}||u||_{k+1}^{2}.
\end{equation*}
Then from the definition of  $\mathcal{I}_{K}$ and $\mathcal{I}_{F}$, we have $\mathcal{I}_{F}p=\mathcal{I}_{K}p|_{\mathcal{F}}\in \bar{Q}_{h}$ and
\begin{equation*}
\sum_{K\in\mathcal{T}_{h}}\left\Vert\left(\frac{\alpha}{\nu}\right)^{\frac{1}{2}}(\eta_{p}-\eta_{\hat{p}})\right\Vert_{\partial K}^{2}=0.
\end{equation*}
Finally, applying \eqref{BB-1} yields
$\gamma||\eta_{p}||^{2}\le Ch^{2(k+1)}||p||_{k+1}^{2}$.

Summing the contributions of all terms and taking the square root completes the proof.
\end{proof}
%
%
%
\section{Error estimates}
\begin{theorem}\label{DDD-2}
Let $(u, p)\in [H^{k+1}(\Omega)]^{d}\times H^{k+1}(\Omega)$ with $k \ge 1$. Then, there exists a constant $C > 0$ independent of $h$ and $\nu$, 
\begin{equation*}
\begin{aligned}
|||({\bm{\xi}}_{u}, {\bm{\xi}}_{q})|||&\le Ch^{k}(||u||_{k+1}+||p||_{k+1}). 
\end{aligned}
\end{equation*}
\end{theorem}
\begin{proof}
Using the Galerkin orthogonality and Lemma \ref{WW-2}, we have
\begin{equation*}
\begin{aligned}
|||({\bm{\xi}}_{u},{\bm{\xi}}_{p})|||&\le \frac{1}{C_{s}}\sup_{({\bm{v}}_{h},{\bm{q}}_{h})\in X_{h}^{*}}\frac{B_{h}(({\bm{\xi}}_{u},{\bm{\xi}}_{p});({\bm{v}}_{h},{\bm{q}}_{h}))}{|||({\bm{v}}_{h},{\bm{q}}_{h})|||}\\
&=\frac{1}{C_{s}}\sup_{({\bm{v}}_{h},{\bm{q}}_{h})\in X_{h}^{*}}\frac{B_{h}(({\bm{e}}_{u}-{\bm{\eta}}_{u}, {\bm{e}}_{p}-{\bm{\eta}}_{p}); ({\bm{v}}_{h},{\bm{q}}_{h}))}{|||({\bm{v}}_{h},{\bm{q}}_{h})|||}\\
&=-\frac{1}{C_{s}}\sup_{({\bm{v}}_{h},{\bm{q}}_{h})\in X_{h}^{*}}\frac{B_{h}(({\bm{\eta}}_{u}, {\bm{\eta}}_{p}); ({\bm{v}}_{h},{\bm{q}}_{h}))}{|||({\bm{v}}_{h},{\bm{q}}_{h})|||}.
\end{aligned}
\end{equation*}
The specific equations are described below:
\begin{equation*}
\begin{aligned}
&B_{h}(({\bm{\eta}}_{u}, {\bm{\eta}}_{p});({\bm{v}}_{h},{\bm{q}}_{h}))\\
&=\nu\sum_{K\in\mathcal{T}_{h}}\int_{K}\nabla\eta_{u}:\nabla v_{h}\mr{d}x-\nu\sum_{K\in\mathcal{T}_{h}}\int_{\partial K}(\eta_{u}-\eta_{\hat{u}})\cdot\frac{\partial v_{h}}{\partial n}\mr{d}s\\
&-\nu\sum_{K\in\mathcal{T}_{h}}\int_{\partial K}\frac{\partial \eta_{u}}{\partial n}\cdot(v_{h}-\bar{v}_{h})\mr{d}s+\nu\sum_{K\in\mathcal{T}_{h}}\int_{\partial K}\frac{\eta}{h_{K}}(\eta_{u}-\eta_{\hat{u}})\cdot(v_{h}-\bar{v}_{h})\mr{d}s\\
&-\sum_{K\in\mathcal{T}_{h}}\int_{K}\eta_{u} b\cdot\nabla v_{h}\mr{d}x+\sum_{K\in\mathcal{T}_{h}}\int_{K}\sigma\eta_{u} v_{h}\mr{d}x+\sum_{K\in\mathcal{T}_{h}}\int_{\partial K}\frac{1}{2}( b\cdot n)(\eta_{u}+\eta_{\hat{u}})\cdot(v_{h}-\bar{v}_{h})\mr{d}s\\
&+\sum_{K\in\mathcal{T}_{h}}\int_{\partial K}\frac{1}{2}| b\cdot n|(\eta_{u}-\eta_{\hat{u}})\cdot(v_{h}-\bar{v}_{h})\mr{d}s-\sum_{K\in\mathcal{T}_{h}}\int_{K}\eta_{p}\nabla\cdot v_{h}\mr{d}x\\
&+\sum_{K\in\mathcal{T}_{h}}\int_{\partial K}\eta_{\hat{p}}v_{h}\cdot n\mr{d}s+\sum_{K\in\mathcal{T}_{h}}\int_{K}\nabla\cdot\eta_{u}q_{h}\mr{d}x-\sum_{K\in\mathcal{T}_{h}}\int_{\partial K}\eta_{u}\cdot n\bar{q}_{h}\mr{d}s\\
&+\sum_{K\in\mathcal{T}_{h}}\int_{\partial K}\frac{\alpha}{\nu}h_{K}(\eta_{p}-\eta_{\hat{p}})\cdot(q_{h}-\bar{q}_{h})\mr{d}s.
\end{aligned}
\end{equation*}

We now analyse the right-hand side terms of the equation in turn,
\begin{equation*}
\begin{aligned}
\nu\sum_{K\in\mathcal{T}_{h}}\int_{K}\nabla\eta_{u}:\nabla v_{h}\mr{d}x&\le C\nu||\nabla\eta_{u}|| ||\nabla v_{h}||&\le C\nu^{\frac{1}{2}}\left(\sum_{K\in\mathcal{T}_{h}} h^{2k}||u||_{k+1,K}^{2}\right)^{\frac{1}{2}}|||({\bm{v}}_{h}, {\bm{q}}_{h})|||.
\end{aligned}
\end{equation*}

By using the interpolation theory \eqref{P-111} and the definition of the norm \eqref{GG-1}, 
\begin{equation*}
\begin{aligned}
&\nu\sum_{K\in\mathcal{T}_{h}}\int_{\partial K}(\eta_{u}-\eta_{\hat{u}})\cdot\frac{\partial v_{h}}{\partial n}\mr{d}s\\
&\le C\nu\left(\sum_{K\in\mathcal{T}_{h}}||\eta_{u}-\eta_{\hat{u}}||_{\partial K}^{2}\right)^{\frac{1}{2}}\left(\sum_{K\in\mathcal{T}_{h}}||\nabla v_{h}\cdot n||^{2}_{\partial K}\right)^{\frac{1}{2}}\\
&\le C\nu h^{-\frac{1}{2}}\left(\sum_{K\in\mathcal{T}_{h}}(||\eta_{u}||_{\partial K}^{2}+||\eta_{\hat{u}}||_{\partial K}^{2})\right)^{\frac{1}{2}}\left(\sum_{K\in\mathcal{T}_{h}}||\nabla v_{h}||^{2}_{K}\right)^{\frac{1}{2}}\\
&\le C\nu^{\frac{1}{2}}\left(\sum_{K\in\mathcal{T}_{h}} h^{2k}||u||_{k+1,K}^{2}\right)^{\frac{1}{2}}|||({\bm{v}}_{h}, {\bm{q}}_{h})|||.
\end{aligned}
\end{equation*}

From the approximation property \citep[Lemma 1.59]{Di1Ern2:2012--M}, one obtains
\begin{equation*}
\begin{aligned}
&\nu\sum_{K\in\mathcal{T}_{h}}\int_{\partial K}\frac{\partial \eta_{u}}{\partial n}\cdot(v_{h}-\bar{v}_{h})\mr{d}s\\
&\le C\nu\left(\frac{h}{\eta}\right)^{\frac{1}{2}}\left(\sum_{K\in\mathcal{T}_{h}}||\nabla\eta_{u}\cdot n||_{\partial K}^{2}\right)^{\frac{1}{2}}\left(\sum_{K\in\mathcal{T}_{h}}\int_{\partial K}\frac{\eta}{h_{K}}(v_{h}-\bar{v}_{h})^{2}\mr{d}s\right)^{\frac{1}{2}}\\
&\le C\nu^{\frac{1}{2}}\left(\sum_{K\in\mathcal{T}_{h}} h^{2k}||u||_{k+1,K}^{2}\right)^{\frac{1}{2}}|||({\bm{v}}_{h}, {\bm{q}}_{h})|||.
\end{aligned}
\end{equation*}

By using the H${\ddot{o}}$lder inequality, \eqref{P-111} and \eqref{eq:interpolation-theory-3}, there is
\begin{equation*}
\begin{aligned}
&\nu\sum_{K\in\mathcal{T}_{h}}\int_{\partial K}\frac{\eta}{h_{K}}(\eta_{u}-\eta_{\hat{u}})\cdot(v_{h}-\bar{v}_{h})\mr{d}s\\
&\le C\nu^{\frac{1}{2}}\left(\sum_{K\in\mathcal{T}_{h}}\frac{\eta}{h_{K}}||\eta_{u}-\eta_{\hat{u}}||^{2}_{\partial K}\right)^{\frac{1}{2}}\left(\sum_{K\in\mathcal{T}_{h}}\frac{\nu\eta}{h_{K}}||v_{h}-\bar{v}_{h}||_{\partial K}^{2}\right)^{\frac{1}{2}}\\
&\le C\nu^{\frac{1}{2}}\left(\sum_{K\in\mathcal{T}_{h}}\frac{\eta}{h_{K}}(||\eta_{u}||^{2}_{\partial K}+||\eta_{\hat{u}}||^{2}_{\partial K})\right)^{\frac{1}{2}}|||({\bm{v}}_{h}, {\bm{q}}_{h})|||\\
&\le C\nu^{\frac{1}{2}}\left(\sum_{K\in\mathcal{T}_{h}} h^{2k}||u||_{k+1,K}^{2}\right)^{\frac{1}{2}}|||({\bm{v}}_{h}, {\bm{q}}_{h})|||.
\end{aligned}
\end{equation*}

According to the H$\ddot{o}$lder inequality and the inverse inequality (see \citep[Lemma 1.28]{Di1Dro2:2020--motified}), 
\begin{equation*}
\begin{aligned}
&\sum_{K\in\mathcal{T}_{h}}\int_{K}\eta_{u} b\cdot\nabla v_{h}\mr{d}x\\&\le C\left(\sum_{K\in\mathcal{T}_{h}}||\eta_{u}||_{K}^{2}\right)^{\frac{1}{2}} \left(\sum_{K\in\mathcal{T}_{h}}||\nabla v_{h}||_{K}^{2}\right)^{\frac{1}{2}}\\
&\le Ch^{-1}\left(\sum_{K\in\mathcal{T}_{h}} h^{2k+2}||u||_{k+1,K}^{2}\right)^{\frac{1}{2}}\left(\sum_{K\in\mathcal{T}_{h}}||v_{h}||_{K}^{2}\right)^{\frac{1}{2}}\\
&\le C\left(\sum_{K\in\mathcal{T}_{h}} h^{2k}||u||_{k+1,K}^{2}\right)^{\frac{1}{2}}|||({\bm{v}}_{h}, {\bm{q}}_{h})|||.
\end{aligned}
\end{equation*}

From the H$\ddot{o}$lder inequality and \eqref{P-111}, there is
\begin{equation*}
\begin{aligned}
&\sum_{K\in\mathcal{T}_{h}}\int_{K}\sigma\eta_{u}v_{h}\mr{d}x\\&\le C\left(\sum_{K\in\mathcal{T}_{h}}||\eta_{u}||_{K}^{2}\right)^{\frac{1}{2}} \left(\sum_{K\in\mathcal{T}_{h}}||v_{h}||_{K}^{2}\right)^{\frac{1}{2}}\\
&\le C\left(\sum_{K\in\mathcal{T}_{h}} h^{2k+2}||u||_{k+1,K}^{2}\right)^{\frac{1}{2}}|||({\bm{v}}_{h},{\bm{q}}_{h})|||.
\end{aligned}
\end{equation*}

Then according to the H$\ddot{o}$lder inequality, \eqref{P-111} and \eqref{eq:interpolation-theory-3}, 
\begin{equation*}
\begin{aligned}
&\sum_{K\in\mathcal{T}_{h}}\int_{\partial K}\frac{1}{2}( b\cdot n)(\eta_{u}+\eta_{\hat{u}})\cdot(v_{h}-\bar{v}_{h})\mr{d}s\\
&\le C\left(\sum_{K\in\mathcal{T}_{h}}||\eta_{u}+\eta_{\hat{u}}||_{\partial K}^{2}\right)^{\frac{1}{2}}\left(\sum_{K\in\mathcal{T}_{h}}\int_{\partial K}| b\cdot n|(v_{h}-\bar{v}_{h})^{2}\mr{d}s\right)^{\frac{1}{2}}\\
&\le C\left(\sum_{K\in\mathcal{T}_{h}}(||\eta_{u}||_{\partial K}^{2}+||\eta_{\hat{u}}||_{\partial K}^{2})\right)^{\frac{1}{2}}|||({\bm{v}}_{h}, {\bm{q}}_{h})|||\\
&\le C\left(\sum_{K\in\mathcal{T}_{h}} h^{2k+1}||u||_{k+1,K}^{2}\right)^{\frac{1}{2}}|||({\bm{v}}_{h},{\bm{q}}_{h})|||,
\end{aligned}
\end{equation*}
and using a similar method,
\begin{equation*}
\begin{aligned}
&\sum_{K\in\mathcal{T}_{h}}\int_{\partial K}\frac{1}{2}| b\cdot n|(\eta_{u}-\eta_{\hat{u}})\cdot(v_{h}-\bar{v}_{h})\mr{d}s\\
&\le C\left(\sum_{K\in\mathcal{T}_{h}}||\eta_{u}-\eta_{\hat{u}}||_{\partial K}^{2}\right)^{\frac{1}{2}} \left(\sum_{K\in\mathcal{T}_{h}}\int_{\partial K}| b\cdot n|(v_{h}-\bar{v}_{h})^{2}\mr{d}s\right)^{\frac{1}{2}}\\
&\le C\left(\sum_{K\in\mathcal{T}_{h}}(||\eta_{u}||_{\partial K}^{2}+||\eta_{\hat{u}}||_{\partial K}^{2})\right)^{\frac{1}{2}}|||({\bm{v}}_{h}, {\bm{q}}_{h})|||\\
&\le C\left(\sum_{K\in\mathcal{T}_{h}} h^{2k+1}||u||_{k+1,K}^{2}\right)^{\frac{1}{2}}|||({\bm{v}}_{h}, {\bm{q}}_{h})|||.
\end{aligned}
\end{equation*}

By means of the H$\ddot{o}$lder inequality, the inverse inequality \citep[Lemma 1.28]{Di1Dro2:2020--motified} and \eqref{BB-1}, 
\begin{equation*}
\begin{aligned}
&\sum_{K\in\mathcal{T}_{h}}\int_{K}\eta_{p}\nabla\cdot v_{h}\mr{d}x\\&\le C\left(\sum_{K\in\mathcal{T}_{h}}||\eta_{p}||_{K}^{2}\right)^{\frac{1}{2}} \left(\sum_{K\in\mathcal{T}_{h}}||\nabla\cdot v_{h}||_{K}^{2}\right)^{\frac{1}{2}}\\
&\le C\left(\sum_{K\in\mathcal{T}_{h}}||\eta_{p}||_{K}^{2}\right)^{\frac{1}{2}}\left(\sum_{K\in\mathcal{T}_{h}}h_{K}^{-2}||v_{h}||_{K}^{2}\right)^{\frac{1}{2}}\\
&\le C\left(\sum_{K\in\mathcal{T}_{h}} h^{2k}||p||_{k+1,K}^{2}\right)^{\frac{1}{2}}|||({\bm{v}}_{h}, {\bm{q}}_{h})|||.
\end{aligned}
\end{equation*}

According to the discrete trace inequality \eqref{ZZ-2}, 
\begin{equation*}
\begin{aligned}
&\sum_{K\in\mathcal{T}_{h}}\int_{\partial K}\eta_{\hat{p}}v_{h}\cdot n\mr{d}s\\&\le C\left(\sum_{K\in\mathcal{T}_{h}}||\eta_{\hat{p}}||_{\partial K}^{2}\right)^{\frac{1}{2}}\left(\sum_{K\in\mathcal{T}_{h}}||v_{h}\cdot n||_{\partial K}^{2}\right)^{\frac{1}{2}}\\
&\le C\left(\sum_{K\in\mathcal{T}_{h}}||\eta_{\hat{p}}||_{\partial K}^{2}\right)^{\frac{1}{2}}\left(\sum_{K\in\mathcal{T}_{h}}h_{K}^{-1}||v_{h}||_{K}^{2}\right)^{\frac{1}{2}}\\
&\le C\left(\sum_{K\in\mathcal{T}_{h}} h^{2k}||p||_{k+1,K}^{2}\right)^{\frac{1}{2}}|||({\bm{v}}_{h}, {\bm{q}}_{h})|||.
\end{aligned}
\end{equation*}

From integration by parts, we can obtain 
\begin{equation*}
\begin{aligned}
&\sum_{K\in\mathcal{T}_{h}}\int_{K}\nabla\cdot\eta_{u}q_{h}\mr{d}x-\sum_{K\in\mathcal{T}_{h}}\int_{\partial K}\eta_{u}\cdot n\bar{q}_{h}\mr{d}s\\
&=-\sum_{K\in\mathcal{T}_{h}}\int_{K}\eta_{u}\cdot\nabla q_{h}\mr{d}x+\sum_{K\in\mathcal{T}_{h}}\int_{\partial K}\eta_{u}\cdot n(q_{h}-\bar{q}_{h})\mr{d}s.
\end{aligned}
\end{equation*}
Here by using the H$\ddot{o}$lder inequality and \eqref{P-111}, one has
\begin{equation*}
\begin{aligned}
&\sum_{K\in\mathcal{T}_{h}}\int_{K}\eta_{u}\cdot\nabla q_{h}\mr{d}x\\&\le C\left(\sum_{K\in\mathcal{T}_{h}}||\eta_{u}||_{K}^{2}\right)^{\frac{1}{2}}\left(\sum_{K\in\mathcal{T}_{h}}||\nabla q_{h}||_{K}^{2}\right)^{\frac{1}{2}}\\
&\le C\left(\sum_{K\in\mathcal{T}_{h}}||\eta_{u}||_{K}^{2}\right)^{\frac{1}{2}}\left(\sum_{K\in\mathcal{T}_{h}}h_{K}^{-2}||q_{h}||_{K}^{2}\right)^{\frac{1}{2}}\\
&\le C\left(\sum_{K\in\mathcal{T}_{h}} h^{2k}||u||_{k+1,K}^{2}\right)^{\frac{1}{2}}|||({\bm{v}}_{h}, {\bm{q}}_{h})|||
\end{aligned}
\end{equation*}
and from \eqref{P-111} and \eqref{GG-1},
\begin{equation*}
\begin{aligned}
&\sum_{K\in\mathcal{T}_{h}}\int_{\partial K}\eta_{u}\cdot n(q_{h}-\bar{q}_{h})\mr{d}s\\
&\le C\left(\frac{\nu}{\alpha}\right)^{\frac{1}{2}}h^{-\frac{1}{2}}\left(\sum_{K\in\mathcal{T}_{h}}||\eta_{u}||^{2}_{\partial K}\right)^{\frac{1}{2}}\left(\sum_{K\in\mathcal{T}_{h}}\int_{\partial K}\frac{\alpha}{\nu}h_{K}(q_{h}-\bar{q}_{h})^{2}\mr{d}s\right)^{\frac{1}{2}}\\
&\le C\nu^{\frac{1}{2}}\left(\sum_{K\in\mathcal{T}_{h}} h^{2k}||u||_{k+1,K}^{2}\right)^{\frac{1}{2}}|||({\bm{v}}_{h}, {\bm{q}}_{h})|||.
\end{aligned}
\end{equation*}

Applying the continuous Lagrange interpolation, we have $(\eta_{p})|_{\mathcal{F}} = \eta_{\hat{p}}$ and
\begin{equation*}
\begin{aligned}
\sum_{K\in\mathcal{T}_{h}}\int_{\partial K}\frac{\alpha}{\nu}h_{K}(\eta_{p}-\eta_{\hat{p}})\cdot(q_{h}-\bar{q}_{h})\mr{d}s=0.
\end{aligned}
\end{equation*}
Therefore, it is straightforward to derive that 
\begin{equation*}
\begin{aligned}
|B_{h}(({\bm{\eta}}_{u}, {\bm{\eta}}_{p});({\bm{v}}_{h},{\bm{q}}_{h}))|\le Ch^{k}(||u||_{k+1}+||p||_{k+1})|||({\bm{v}}_{h}, {\bm{q}}_{h})|||
\end{aligned}
\end{equation*}
and
\begin{equation*}
|||({\bm{\xi}}_{u}, {\bm{\xi}}_{p})|||\le Ch^{k}(||u||_{k+1}+||p||_{k+1}).
\end{equation*}
\end{proof}
\begin{theorem}\label{PPP-1}
Let $(u, p)\in [H^{k+1}(\Omega)]^{d}\times H^{k+1}(\Omega)$ with $k \ge 1$. Then, there exists a constant $C > 0$, independent of the mesh size $h$ and the viscosity $\nu$, such that 
\begin{equation*}
\begin{aligned}
|||({\bm{e}}_{u}, {\bm{e}}_{p})|||\le Ch^{k}(||u||_{k+1}+||p||_{k+1}). 
\end{aligned}
\end{equation*}
The constant $C$ may depend on the reaction coefficient $\sigma$, convection coefficient $b$, penalty parameter $\eta$, pressure penalty parameters $\alpha$ and $\gamma$.
\end{theorem}
\begin{proof}
By the triangle inequality, Theorem \ref{DDD-4}, and Theorem \ref{DDD-2}, the proof is complete.
\end{proof}

\section{Numerical experiment}
We present a numerical experiment to validate the theoretical convergence results established in Theorem \ref{PPP-1}.

\subsection{Experimental Setup}
The computational domain is taken as $\Omega = (0,1)^2$. The exact solutions are selected as:
\begin{equation*}
\begin{aligned}
&u_{1}(x,y)= \sin(\pi x)^{2} \sin(2\pi y), \\
&u_{2}(x,y)= -\sin(2\pi x)\sin(\pi y)^{2}, \\
&p(x,y)= \sin(2\pi x)\sin(2\pi y).
\end{aligned}
\end{equation*}
The convective field is set as $b = (1,0)^{T}$ with reaction coefficient $\sigma = 1$. The source term $f = (f_{1}, f_{2})$ is constructed by substituting the exact solutions into the Oseen equation, ensuring that $(u, p)$ satisfies the equations exactly. Homogeneous Dirichlet boundary conditions are imposed, i.e., $u = 0$ on $\partial\Omega$.

To examine the algorithm's adaptability to different viscosity parameters, we consider two cases: $\nu = 1$ and $\nu = 0.1$. Computations are performed on uniformly refined triangular meshes with mesh sizes $h = 1/2,1/4, \cdots, 1/128, 1/256$. Finite element spaces with polynomial degrees $k = 1$ and $k = 2$ are employed.

\subsection{Stabilization Parameter Selection}\label{OOO-1}
The numerical performance of hybridized discontinuous Galerkin methods is influenced by the stabilization parameters for the diffusion and pressure terms. Following established interior penalty practices, the diffusion stabilization parameter can be chosen as follows: for HDG, $\eta = 6k^{2}$ in 2D and $\eta = 10k^{2}$ in 3D; for EDG and E-HDG, $\eta = 4k^{2}$ in 2D and $\eta = 6k^{2}$ in 3D \cite{Rhe1Well2:2020--motified, Riv:2008--D}. These values have proven to be reliable across a variety of problems.

In our parameter study, we use E-HDG with $k = 1$ as a representative case. All calculations are carried out by using MATLAB. Based on a comprehensive parameter study examining the velocity, pressure, and combined errors (see Figures \ref{HH-1}-\ref{HH-3}), we select the pressure penalty parameter as $\alpha = 10^{-2}$. For other choices of the stabilization parameter $\alpha$, the situation does not improve significantly. 

\begin{figure}[H]
\centering
\hspace*{-1.2cm}\includegraphics[width=130mm, height=80mm]{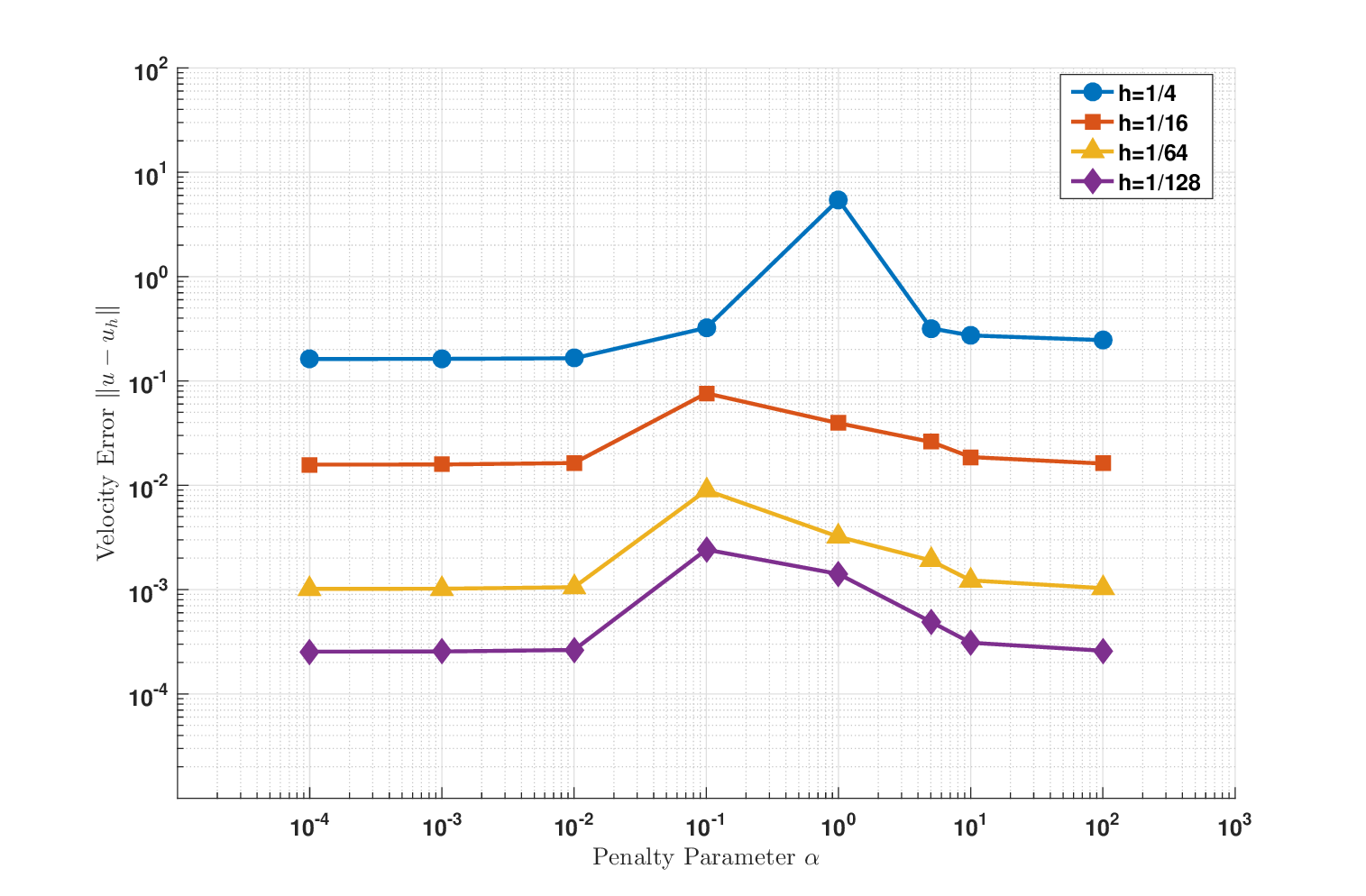}
\vspace{-0.5cm}
\begin{center}
\caption{$\Vert u-u_{h}\Vert$ and pressure penalty parameter $\alpha$ for E-HDG ($\nu=1$, $k=1$, $\eta=4$)}\label{HH-1}
\end{center}
\end{figure}
\vspace{-0.9cm}

\begin{figure}[H]
\centering
\hspace*{-1.2cm}\includegraphics[width=130mm, height=80mm]{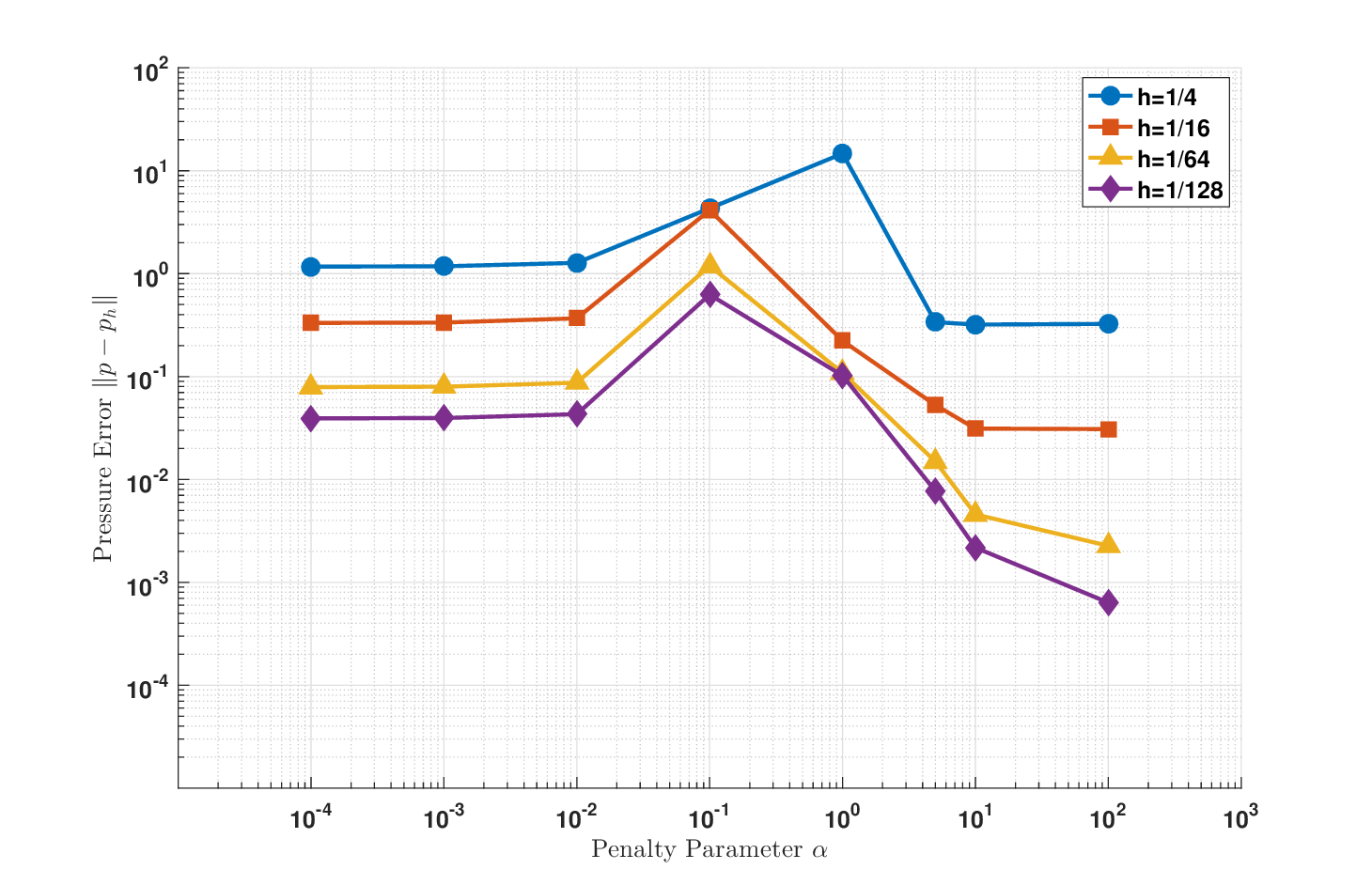}
\vspace{-0.5cm}
\begin{center}
\caption{$\Vert p-p_{h}\Vert$ and pressure penalty parameter $\alpha$ for E-HDG ($\nu=1$, $k=1$, $\eta=4$)}\label{HH-2}
\end{center}
\end{figure}
\vspace{-0.6cm}

\begin{figure}[H]
\centering
\hspace*{-1.2cm}\includegraphics[width=130mm, height=80mm]{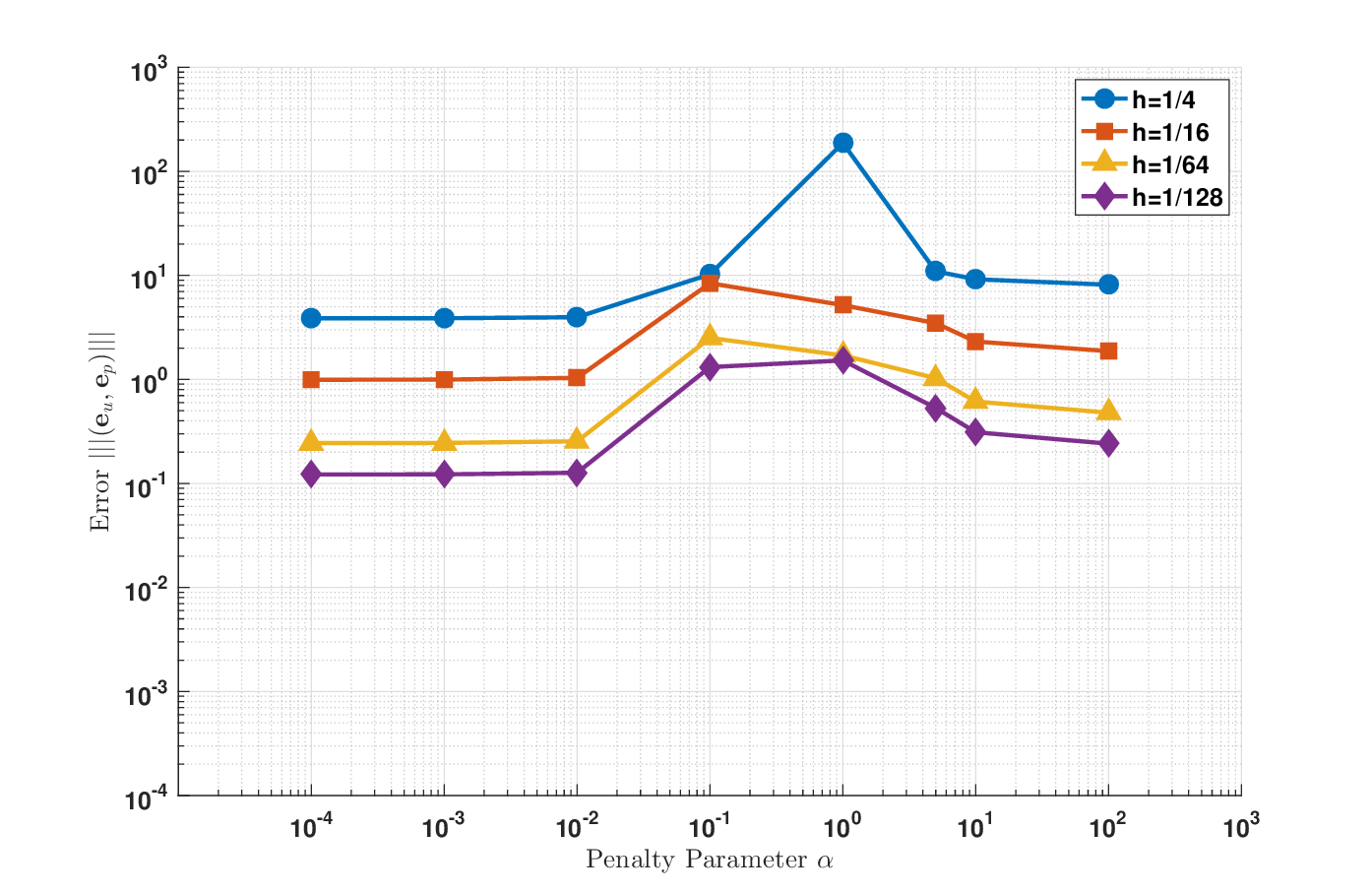}
\vspace{-0.5cm}
\begin{center}
\caption{$|||(\mathbf{e}_{u}, \mathbf{e}_{p})|||$ and pressure penalty parameter $\alpha$ for E-HDG  ($\nu=1$, $k=1$, $\eta=4$)}\label{HH-3}
\end{center}
\end{figure}
\vspace{-0.6cm}
For cases where optimal parameters are not readily identified a priori, we extend the methodology to simultaneously vary both $\eta$ and $\alpha$ over an appropriate parametric range to determine the optimal combination.

\subsection{Observed convergence rates}
For each computation, we record velocity errors in the $L^{2}$ norm, pressure errors in the $L^{2}$ norm, and combined errors in the energy norm \eqref{GG-1}. The convergence rate $r$ is computed by
\begin{equation*}
r = \frac{\log(e(h)/e(h/2))}{\log(2)},
\end{equation*}
where $e(h)$ denotes the error at mesh size $h$.

\begin{table}[H]
\caption{Errors and convergence rates in case of $\nu=1$ and $k=1$ for E-HDG method }
\footnotesize
\resizebox{137mm}{20mm}{
\begin{tabular*}{\textwidth}{@{\extracolsep{\fill}} c cccccc}
\cline{1-7}{}
            \multirow{2}{*}{$h$}   &\multicolumn{2}{c}{$||u-u_{h}||$} &\multicolumn{2}{c}{$||p-p_{h}||$}  &\multicolumn{2}{c}{$|||({\bm{e}}_{u}, {\bm{e}}_{p})|||$}  \\
 
\cline{2-7}&$e_{u}$&$r_{u}$&$e_{p}$&$r_{p}$&$e_{DG}$&$r_{DG}$\\
\cline{1-7}

             $1/2$   &3.388e-1  &---      &1.022e+0  &---  &5.685e+0  &--- \\
             $1/4$   &1.648e-1  & 1.04  &1.273e+0 &-0.32  &3.965e+0  & 0.52 \\
             $1/8$    &5.858e-2  &1.49   &7.615e-1  &0.74  &2.102e+0  &0.92 \\
             $1/16$   &1.632e-2  &1.84   &3.678e-1  &1.05  &1.035e+0  &1.02  \\
             $1/32$   &4.192e-3  &1.96  &1.774e-1  &1.05  &5.112e-1  &1.02 \\
             $1/64$   &1.054e-3  &1.99  &8.701e-2  &1.03  &2.542e-1  &1.01 \\
             $1/128$   &2.638e-4  &2.00  &4.309e-2  &1.01  &1.268e-1  &1.00 \\
             $1/256$   &6.595e-5  &2.00   &2.144e-2  &1.01  &6.335e-2 &1.00\\
\cline{1-7}
\end{tabular*}}
\label{table:1-1}
\end{table}

\vspace{-0.4cm}
\begin{table}[H]
\caption{Errors and convergence rates in case of $\nu=1$ and $k=1$ for EDG method }
\footnotesize
\resizebox{137mm}{20mm}{
\begin{tabular*}{\textwidth}{@{\extracolsep{\fill}} c cccccc}
\cline{1-7}{}
            \multirow{2}{*}{ $h$ }   &\multicolumn{2}{c}{$||u-u_{h}||$} &\multicolumn{2}{c}{$||p-p_{h}||$}  &\multicolumn{2}{c}{$|||({\bm{e}}_{u}, {\bm{e}}_{p})|||$}  \\
 
\cline{2-7}&$e_{u}$&$r_{u}$&$e_{p}$&$r_{p}$&$e_{DG}$&$r_{DG}$\\
\cline{1-7}

             $1/2$   &3.981e-1  &---      &5.790e-1  &---  & 5.734e+0  &--- \\
             $1/4$   &1.610e-1  &1.31  &4.170e-1 &0.47  &4.561e+0  &0.33 \\
             $1/8$    &4.783e-2  &1.75   &1.933e-1  &1.11  &2.422e+0  &0.91 \\
             $1/16$   &1.254e-2  &1.93   &9.256e-2  &1.06  &1.234e+0  &0.97  \\
             $1/32$   &3.174e-3  &1.98  &4.568e-2  &1.02 &6.213e-1  &0.99 \\
             $1/64$   &7.962e-4  &1.99  &2.278e-2  &1.00  &3.115e-1 &1.00 \\
             $1/128$   &1.993e-4  &2.00  &1.139e-2  &1.00  &1.560e-1  &1.00 \\
             $1/256$   &4.984e-5 &2.00   &5.693e-3 &1.00  &7.803e-2   &1.00\\
\cline{1-7}
\end{tabular*}}
\label{table:1-2}
\end{table}

\vspace{-0.5cm}
\begin{table}[H]
\caption{Errors and convergence rates in case of $\nu=1$ and $k=1$ for HDG method }
\footnotesize
\resizebox{137mm}{20mm}{
\begin{tabular*}{\textwidth}{@{\extracolsep{\fill}} c cccccc}
\cline{1-7}{}
            \multirow{2}{*}{ $h$ }   &\multicolumn{2}{c}{$||u-u_{h}||$} &\multicolumn{2}{c}{$||p-p_{h}||$}  &\multicolumn{2}{c}{$|||({\bm{e}}_{u}, {\bm{e}}_{p})|||$}  \\
 
\cline{2-7}&$e_{u}$&$r_{u}$&$e_{p}$&$r_{p}$&$e_{DG}$&$r_{DG}$\\
\cline{1-7}

             $1/2$   &2.909e-1  &---      &1.729e+0  &---  &5.665e+0  &--- \\
             $1/4$   &1.275e-1  &1.19  &1.906e+0 &-0.14  &3.925e+0  &0.53 \\
             $1/8$    &4.298e-2  &1.57   &1.304e+0  &0.55  &2.342e+0  &0.74 \\
             $1/16$   &1.212e-2  &1.83   &7.380e-1  &0.82  &1.249e+0  &0.91  \\
             $1/32$   &3.141e-3  &1.95  &3.834e-1  &0.94  &6.368e-1  &0.97 \\
             $1/64$   &7.922e-4  &1.99  &1.937e-1  &0.98 &3.201e-1  &0.99 \\
             $1/128$   &1.984e-4  &2.00  &9.713e-2  &1.00 &1.603e-1  &1.00 \\
             $1/256$   &4.962e-5  &2.00   &4.861e-2  &1.00 &8.016e-2 &1.00\\
\cline{1-7}
\end{tabular*}}
\label{table:1-3}
\end{table}
\vspace{-0.5cm}
\begin{table}[H]
\caption{Errors and convergence rates in case of $\nu=1$ and $k=2$ for E-HDG method }
\footnotesize
\resizebox{140mm}{17mm}{
\begin{tabular*}{\textwidth}{@{\extracolsep{\fill}} c cccccc}
\cline{1-7}{}
            \multirow{2}{*}{$h$}   &\multicolumn{2}{c}{$||u-u_{h}||$} &\multicolumn{2}{c}{$||p-p_{h}||$}  &\multicolumn{2}{c}{$|||({\bm{e}}_{u}, {\bm{e}}_{p})|||$}  \\
 
\cline{2-7}&$e_{u}$&$r_{u}$&$e_{p}$&$r_{p}$&$e_{DG}$&$r_{DG}$\\
\cline{1-7}

             $1/2$   &2.338e-1  &---      &6.460e-1  &---  &4.382e+0  &--- \\
             $1/4$   &3.147e-2  &2.89  &4.149e-1 &0.64  &1.422e+0  &1.62 \\
             $1/8$    &3.736e-3  &3.07   &1.090e-1  &1.93  &3.847e-1  &1.89 \\
             $1/16$   &4.185e-4  &3.16   &2.442e-2  &2.16  &9.354e-2  &2.04  \\
             $1/32$   &4.860e-5  &3.11  &5.371e-3  &2.18  &2.263e-2  &2.05 \\
             $1/64$   &5.910e-6  &3.04  &1.263e-3 &2.09  &5.586e-3  &2.02 \\
\cline{1-7}
\end{tabular*}}
\label{table:1-4}
\end{table}

\begin{table}[H]
\caption{Errors and convergence rates in case of $\nu=1$ and $k=2$ for EDG method }
\footnotesize
\resizebox{140mm}{17mm}{
\begin{tabular*}{\textwidth}{@{\extracolsep{\fill}} c cccccc}
\cline{1-7}{}
            \multirow{2}{*}{ $h$ }   &\multicolumn{2}{c}{$||u-u_{h}||$} &\multicolumn{2}{c}{$||p-p_{h}||$}  &\multicolumn{2}{c}{$|||({\bm{e}}_{u}, {\bm{e}}_{p})|||$}  \\
 
\cline{2-7}&$e_{u}$&$r_{u}$&$e_{p}$&$r_{p}$&$e_{DG}$&$r_{DG}$\\
\cline{1-7}

             $1/2$   &1.221e+0  &---      &6.406e+0  &---  & 1.467e+1  &--- \\
             $1/4$   &9.148e-2  &3.74  &1.104e+0 &2.54  &2.398e+0  &2.61 \\
             $1/8$    &1.044e-2  &3.13   &2.348e-1  &2.23  &5.388e-1  &2.15 \\
             $1/16$   &1.277e-3  &3.03   &5.936e-2  &1.98  &1.328e-1  &2.02  \\
             $1/32$   &1.559e-4  &3.03  &1.513e-2  &1.97 &3.290e-2  &2.01 \\
             $1/64$   &1.929e-5  &3.01  &3.816e-3  &1.99  &8.193e-3 &2.01 \\
\cline{1-7}
\end{tabular*}}
\label{table:1-5}
\end{table}

\begin{table}[H]
\caption{Errors and convergence rates in case of $\nu=1$ and $k=2$ for HDG method }
\footnotesize
\resizebox{137mm}{18mm}{
\begin{tabular*}{\textwidth}{@{\extracolsep{\fill}} c cccccc}
\cline{1-7}{}
            \multirow{2}{*}{ $h$ }   &\multicolumn{2}{c}{$||u-u_{h}||$} &\multicolumn{2}{c}{$||p-p_{h}||$}  &\multicolumn{2}{c}{$|||({\bm{e}}_{u}, {\bm{e}}_{p})|||$}  \\
 
\cline{2-7}&$e_{u}$&$r_{u}$&$e_{p}$&$r_{p}$&$e_{DG}$&$r_{DG}$\\
\cline{1-7}

             $1/2$   &2.404e-1  &---      &3.565e+0  &---  &4.660e+0  &--- \\
             $1/4$   &3.347e-2  &2.84  &1.197e+0 &1.57  &1.604e+0  &1.54 \\
             $1/8$    &3.631e-3  &3.20   &3.471e-1  &1.79 &4.570e-1  &1.81 \\
             $1/16$   &3.934e-4  &3.21   &9.083e-2  &1.93  &1.188e-1  &1.94  \\
             $1/32$   &4.627e-5  &3.09  &2.285e-2  &1.99  &2.990e-2  &1.99 \\
             $1/64$   &5.672e-6  &3.03  &5.695e-3  &2.00 &7.463e-3  &2.00 \\
             $1/128$   &7.049e-7 &3.01 &1.419e-3 &2.00 &1.862e-3   &2.00 \\
\cline{1-7}
\end{tabular*}}
\label{table:1-6}
\end{table}

As shown in Tables \ref{table:1-1}-\ref{table:1-6}, all methods achieve the optimal convergence rate of $\mathcal{O}(h^{k+1})$ for the velocity in the $L^{2}$-norm. Meanwhile, the pressure in the $L^2$-norm and the velocity-pressure pair under the norm \eqref{GG-1} both converge at the expected rate of $\mathcal{O}(h^{k})$. In particular, from Tables \ref{table:1-7}-\ref{table:1-12}, the methods preserve ideal convergence behavior even for $\nu=0.1$. 

\begin{table}[H]
\caption{Errors and convergence rates in case of $\nu=0.1$ and $k=1$ for E-HDG method }
\footnotesize
\resizebox{137mm}{20mm}{
\begin{tabular*}{\textwidth}{@{\extracolsep{\fill}} c cccccc}
\cline{1-7}{}
            \multirow{2}{*}{ $h$ }   &\multicolumn{2}{c}{$||u-u_{h}||$} &\multicolumn{2}{c}{$||p-p_{h}||$}  &\multicolumn{2}{c}{$|||({\bm{e}}_{u}, {\bm{e}}_{p})|||$}  \\
 
\cline{2-7}&$e_{u}$&$r_{u}$&$e_{p}$&$r_{p}$&$e_{DG}$&$r_{DG}$\\
\cline{1-7}

             $1/2$   &4.996e-1  &---      &5.195e-1  &---  &2.387e+0  &--- \\
             $1/4$   &2.021e-1  &1.31  &2.096e-1 &1.31  &1.573e+0  &0.60 \\
             $1/8$    &5.067e-2  &2.00   &3.895e-2  &2.43  &8.454e-1  &0.90 \\
             $1/16$   &1.373e-2  &1.88   &8.723e-3  &2.16  &5.088e-1  &0.73  \\
             $1/32$   &4.138e-3  &1.73  &2.500e-3  &1.80  &3.311e-1  &0.62 \\
             $1/64$   &1.285e-3  &1.69  &1.080e-3  &1.21 &.2.142e-1  &0.63 \\
             $1/128$   &3.846e-4  &1.74  &6.033e-4  &0.84  &1.306e-1  &0.71 \\
             $1/256$   &1.085e-4  &1.83   &3.394e-4  &0.83  &7.425e-2 &0.81\\
\cline{1-7}
\end{tabular*}}
\label{table:1-7}
\end{table}

\begin{table}[H]
\caption{Errors and convergence rates in case of $\nu=0.1$ and $k=1$ for EDG method }
\footnotesize
\resizebox{137mm}{20mm}{
\begin{tabular*}{\textwidth}{@{\extracolsep{\fill}} c cccccc}
\cline{1-7}{}
            \multirow{2}{*}{ $h$ }   &\multicolumn{2}{c}{$||u-u_{h}||$} &\multicolumn{2}{c}{$||p-p_{h}||$}  &\multicolumn{2}{c}{$|||({\bm{e}}_{u}, {\bm{e}}_{p})|||$}  \\
 
\cline{2-7}&$e_{u}$&$r_{u}$&$e_{p}$&$r_{p}$&$e_{DG}$&$r_{DG}$\\
\cline{1-7}

             $1/2$   &5.005e-1  &---      &5.411e-1  &---  &2.304e+0  &--- \\
             $1/4$   &1.885e-1  &1.41  &2.098e-1 &1.37  &1.387e+0  &0.73 \\
             $1/8$    &4.508e-2  &2.06   &3.996e-2  &2.39  &7.060e-1  &0.97 \\
             $1/16$   &1.103e-2  &2.03   &1.142e-2  &1.81  &3.644e-1  &0.95 \\
             $1/32$   &2.765e-3  &2.00  &4.703e-3  &1.28  &1.880e-1  &0.95 \\
             $1/64$   &6.956e-4  &1.99  &2.259e-3  &1.06 &9.602e-2  &0.97 \\
             $1/128$   &1.748e-4  &1.99  &1.127e-3  &1.00  &4.860e-2  &0.98 \\
             $1/256$   &4.383e-5  &2.00   &5.656e-4  &0.99  &2.446e-2 &0.99\\
\cline{1-7}
\end{tabular*}}
\label{table:1-8}
\end{table}

\begin{table}[H]
\caption{Errors and convergence rates in case of $\nu=0.1$ and $k=1$ for HDG method }
\footnotesize
\resizebox{137mm}{20mm}{
\begin{tabular*}{\textwidth}{@{\extracolsep{\fill}} c cccccc}
\cline{1-7}{}
            \multirow{2}{*}{ $h$ }   &\multicolumn{2}{c}{$||u-u_{h}||$} &\multicolumn{2}{c}{$||p-p_{h}||$}  &\multicolumn{2}{c}{$|||({\bm{e}}_{u}, {\bm{e}}_{p})|||$}  \\
 
\cline{2-7}&$e_{u}$&$r_{u}$&$e_{p}$&$r_{p}$&$e_{DG}$&$r_{DG}$\\
\cline{1-7}

             $1/2$   &3.027e-1  &---      &3.447e-1  &---  &1.832e+0  &--- \\
             $1/4$   &1.252e-1  &1.27  &2.385e-1 &0.53  &1.148e+0  &0.67 \\
             $1/8$    &3.910e-2  &1.68   &1.431e-1  &0.74  &6.424e-1  &0.84 \\
             $1/16$   &1.050e-2  &1.90   &7.656e-2  &0.90  &3.297e-1  &0.96 \\
             $1/32$   &2.660e-3  &1.98  &3.891e-2  &0.98  &1.657e-1  &0.99 \\
             $1/64$   &6.649e-4  &2.00 &1.950e-2  &1.00 &8.293e-2  &1.00 \\
             $1/128$  &1.659e-4  &2.00  &9.742e-3  &1.00  &4.147e-2  &1.00 \\
             $1/256$  &4.141e-5 &2.00  &4.867e-3   &1.00  &2.074e-2  &1.00\\
\cline{1-7}
\end{tabular*}}
\label{table:1-9}
\end{table}


\begin{table}[H]
\caption{Errors and convergence rates in case of $\nu=0.1$ and $k=2$ for E-HDG method }
\footnotesize
\resizebox{140mm}{17mm}{
\begin{tabular*}{\textwidth}{@{\extracolsep{\fill}} c cccccc}
\cline{1-7}{}
            \multirow{2}{*}{ $h$ }   &\multicolumn{2}{c}{$||u-u_{h}||$} &\multicolumn{2}{c}{$||p-p_{h}||$}  &\multicolumn{2}{c}{$|||({\bm{e}}_{u}, {\bm{e}}_{p})|||$}  \\
 
\cline{2-7}&$e_{u}$&$r_{u}$&$e_{p}$&$r_{p}$&$e_{DG}$&$r_{DG}$\\
\cline{1-7}

             $1/2$   &1.226e+0  &---      &1.262e+0  &---  &9.660e+0  &--- \\
             $1/4$   &3.834e-1  &1.68  &5.651e-1 &1.16  &5.727e+0  &0.75 \\
             $1/8$    &3.439e-3  &6.80   &1.153e-2  &5.62  &1.148e-1  &5.64 \\
             $1/16$   &3.954e-4  &3.12   &2.516e-3  &2.20  &2.799e-2  &2.04  \\
             $1/32$   &4.728e-5  &3.06  &5.545e-4  &2.18  &6.885e-3  &2.02 \\
             $1/64$   &5.832e-6 &3.02  &1.291e-4  &2.10 &1.714e-3  &2.01\\
\cline{1-7}
\end{tabular*}}
\label{table:1-10}
\end{table}

\begin{table}[H]
\caption{Errors and convergence rates in case of $\nu=0.1$ and $k=2$ for EDG method}
\footnotesize
\resizebox{140mm}{17mm}{
\begin{tabular*}{\textwidth}{@{\extracolsep{\fill}} c cccccc}
\cline{1-7}{}
            \multirow{2}{*}{ $h$ }   &\multicolumn{2}{c}{$||u-u_{h}||$} &\multicolumn{2}{c}{$||p-p_{h}||$}  &\multicolumn{2}{c}{$|||({\bm{e}}_{u}, {\bm{e}}_{p})|||$}  \\
 
\cline{2-7}&$e_{u}$&$r_{u}$&$e_{p}$&$r_{p}$&$e_{DG}$&$r_{DG}$\\
\cline{1-7}

             $1/2$   &2.587e-1  &---      &4.569e-1  &---  &1.097e+0  &--- \\
             $1/4$   &3.297e-2  &2.97  &1.215e-1 &1.91  &3.541e-1  &1.63 \\
             $1/8$    &4.136e-3  &2.99   &4.232e-2  &1.52  &1.072e-1  &1.72 \\
             $1/16$   &4.949e-4  &3.06   &1.337e-2  &1.66  &2.959e-2  &1.86 \\
             $1/32$   &5.956e-5  &3.05  &3.675e-3  &1.86  &7.657e-3  &1.95 \\
             $1/64$   &7.335e-6  &3.02  &9.446e-4  &1.96 &1.932e-3  &1.99\\
\cline{1-7}
\end{tabular*}}
\label{table:1-11}
\end{table}

\begin{table}[H]
\caption{Errors and convergence rates in case of $\nu=0.1$ and $k=2$ for HDG method }
\footnotesize
\resizebox{140mm}{17mm}{
\begin{tabular*}{\textwidth}{@{\extracolsep{\fill}} c cccccc}
\cline{1-7}{}
            \multirow{2}{*}{ $h$ }   &\multicolumn{2}{c}{$||u-u_{h}||$} &\multicolumn{2}{c}{$||p-p_{h}||$}  &\multicolumn{2}{c}{$|||({\bm{e}}_{u}, {\bm{e}}_{p})|||$}  \\
 
\cline{2-7}&$e_{u}$&$r_{u}$&$e_{p}$&$r_{p}$&$e_{DG}$&$r_{DG}$\\
\cline{1-7}

             $1/2$   &2.244e-1  &---      &4.514e-1  &---  &1.122e+0  &--- \\
             $1/4$   &3.040e-2  &2.88  &1.287e-1 &1.81  &3.691e-1  &1.60\\
             $1/8$    &3.403e-3  &3.16  &3.543e-2  &1.86  &1.011e-1  &1.87 \\
             $1/16$   &3.836e-4  &3.15   &9.156e-3  &1.95  &2.596e-2  &1.96 \\
             $1/32$   &4.595e-5  &3.06  &2.293e-3  &2.00 &6.520e-3  &1.99\\
             $1/64$   &5.664e-6  &3.02 &5.705e-4  &2.01 &1.629e-3  &2.00 \\
\cline{1-7}
\end{tabular*}}
\label{table:1-12}
\end{table}
In summary, the numerical results validate the main conlusion and confirm the efficacy of  hybridized discontinuous Galerkin methods for the Oseen equations, even in convection-dominated regimes.

\section{Statements and Declarations}
\subsection{Funding}
This work was supported by the National Natural Science Foundation of China (11771257) and Shandong Provincial Natural Science Foundation, China (ZR2023YQ002).
\subsection{CRediT authorship contribution statement}
{\textbf{Xiaoqi Ma:}} Writing---original draft, Methodology, Investigation, Visualization. {\textbf{Jin Zhang:}} Methodology, Investigation, Funding acquisition.
\subsection{Data availability statement}
The authors confirm that the data supporting the findings of this study are available within the article and its supplementary materials.
\subsection{Conflict of interests}
The authors declare that they have no conflict of interest.


\begin{thebibliography}{10}

\bibitem{Bur1Fer2:2006--C}
E.~Burman, M.~A. Fern\'{a}ndez, and P.~Hansbo.
\newblock Continuous interior penalty finite element method for {O}seen's
  equations.
\newblock {\em SIAM J. Numer. Anal.}, 44(3):1248--1274, 2006.

\bibitem{Gar1Joh2:2021-motified}
B.~Garc\'{\i}a-Archilla, V.~John, and J.~Novo.
\newblock On the convergence order of the finite element error in the kinetic
  energy for high {R}eynolds number incompressible flows.
\newblock {\em Comput. Methods Appl. Mech. Engrg.}, 385:114032, 54,
  2021.

\bibitem{Coc1Kan2:2002--L}
B.~Cockburn, G.~Kanschat, D.~Sch\"{o}tzau, and C.~Schwab.
\newblock Local discontinuous {G}alerkin methods for the {S}tokes system.
\newblock {\em SIAM J. Numer. Anal.}, 40(1):319--343, 2002.


\bibitem{Coc1Kan2:2004--motified}
B. Cockburn, G. Kanschat, and D. Sch\"{o}tzau.
\newblock The local discontinuous {G}alerkin method for the {O}seen
    equations
\newblock {\em Math. Comp.}, 73:569--593, 2004.

\bibitem{Akb1Lin2:2018--motified}
M.~Akbas, A.~Linke, L.~G. Rebholz, and P.~W. Schroeder.
\newblock The analogue of grad-div stabilization in {DG} methods for
  incompressible flows: limiting behavior and extension to tensor-product
  meshes.
\newblock {\em Comput. Methods Appl. Mech. Engrg.}, 341:917--938, 2018.

\bibitem{Rhe1Well2:2017--A}
S.~Rhebergen and G.~N. Wells.
\newblock Analysis of a hybridized/interface stabilized finite element method
  for the {S}tokes equations.
\newblock {\em SIAM J. Numer. Anal.}, 55(4):1982--2003, 2017.

\bibitem{Kir1Rhe2:2019--A}
K.~L.~A. Kirk and S.~Rhebergen.
\newblock Analysis of a pressure-robust hybridized discontinuous {G}alerkin
  method for the stationary {N}avier-{S}tokes equations.
\newblock {\em J. Sci. Comput.}, 81(2):881--897, 2019.

\bibitem{Coc1Gop2:2009--motified}
B.~Cockburn and J.~Gopalakrishnan.
\newblock The derivation of hybridizable discontinuous {G}alerkin methods for
  {S}tokes flow.
\newblock {\em SIAM J. Numer. Anal.}, 47(2):1092--1125, 2009.

\bibitem{Coc1Gop2Ngu3:2011--A}
B.~Cockburn, J.~Gopalakrishnan, N.~C. Nguyen, J.~Peraire, and F.-J. Sayas.
\newblock Analysis of {HDG} methods for {S}tokes flow.
\newblock {\em Math. Comp.}, 80(274):723--760, 2011.

\bibitem{Coc1Say2:2014--D}
B.~Cockburn and F.-J. Sayas.
\newblock Divergence-conforming {HDG} methods for {S}tokes flows.
\newblock {\em Math. Comp.}, 83(288):1571--1598, 2014.

\bibitem{Guz1Coc2Sto3:2007--motified}
S.~G\"{u}zey, B.~Cockburn, and H.~K. Stolarski.
\newblock The embedded discontinuous {G}alerkin method: application to linear
  shell problems.
\newblock {\em Internat. J. Numer. Methods Engrg.}, 70(7):757--790, 2007.

\bibitem{Lab1Well2:2007--motified}
R.~J. Labeur and G.~N. Wells.
\newblock A {G}alerkin interface stabilisation method for the
  advection-diffusion and incompressible {N}avier-{S}tokes equations.
\newblock {\em Comput. Methods Appl. Mech. Engrg.}, 196(49-52):4985--5000,
  2007.

\bibitem{Han1Hou2:2021--motified}
Y.~Han and Y.~Hou.
\newblock An embedded discontinuous {G}alerkin method for the {O}seen
  equations.
\newblock {\em ESAIM Math. Model. Numer. Anal.}, 55(5):2349--2364, 2021.

\bibitem{Lab1Well2:2012--E}
R.~J. Labeur and G.~N. Wells.
\newblock Energy stable and momentum conserving hybrid finite element method
  for the incompressible {N}avier-{S}tokes equations.
\newblock {\em SIAM J. Sci. Comput.}, 34(2):A889--A913, 2012.


\bibitem{Rhe1Well2:2020--motified}
S.~Rhebergen and G.~N. Wells.
\newblock An embedded-hybridized discontinuous {G}alerkin finite element method
  for the {S}tokes equations.
\newblock {\em Comput. Methods Appl. Mech. Engrg.}, 358:112619, 18, 2020.

\bibitem{Bre1For2:1991-M}
F.~Brezzi and M.~Fortin.
\newblock {\em Mixed and Hybrid Finite Element Methods}, volume~15 of {\em
  Springer Series in Computational Mathematics}.
\newblock Springer-Verlag, New York, 1991.

\bibitem{Hug1Fra2:1986-motified}
T.~J.~R. Hughes, L.~P. Franca, and M.~Balestra.
\newblock A new finite element formulation for computational fluid dynamics: V.
  circumventing the Babu{\v{s}}ka-Brezzi condition: A stable Petrov-Galerkin
  formulation of the stokes problem accommodating equal-order interpolations.
\newblock {\em Comput. Methods Appl. Mech. Engrg.},
  59(1):85--99, 1986.
  

\bibitem{Fra1Hug2:1993-C}
L.~P. Franca and T.~J.~R. Hughes.
\newblock Convergence analyses of {G}alerkin least-squares methods for
  symmetric advective-diffusive forms of the {S}tokes and incompressible
  {N}avier-{S}tokes equations.
\newblock {\em Comput. Methods Appl. Mech. Engrg.}, 105(2):285--298, 1993.

\bibitem{Ahm1Cha2:2017--A}
N.~Ahmed, T.~Chac\'{o}n~Rebollo, V.~John, and S.~Rubino.
\newblock Analysis of a full space-time discretization of the {N}avier-{S}tokes
  equations by a local projection stabilization method.
\newblock {\em IMA J. Numer. Anal.}, 37(3):1437--1467, 2017.

\bibitem{de1Gar2Joh3:2019--E}
J.~de~Frutos, B.~Garc\'{\i}a-Archilla, V.~John, and J.~Novo.
\newblock Error analysis of non inf-sup stable discretizations of the
  time-dependent {N}avier-{S}tokes equations with local projection
  stabilization.
\newblock {\em IMA J. Numer. Anal.}, 39(4):1747--1786, 2019.

\bibitem{Gan1Mat2Tob3:2008--L}
S.~Ganesan, G.~Matthies, and L.~Tobiska.
\newblock Local projection stabilization of equal order interpolation applied
  to the {S}tokes problem.
\newblock {\em Math. Comp.}, 77(264):2039--2060, 2008.

\bibitem{Bur1Fer2:2007--C}
E.~Burman and M.~A. Fern\'{a}ndez.
\newblock Continuous interior penalty finite element method for the
  time-dependent {N}avier-{S}tokes equations: space discretization and
  convergence.
\newblock {\em Numer. Math.}, 107(1):39--77, 2007.

\bibitem{Bur1Han2:2006--E}
E.~Burman and P.~Hansbo.
\newblock Edge stabilization for the generalized {S}tokes problem: a continuous
  interior penalty method.
\newblock {\em Comput. Methods Appl. Mech. Engrg.}, 195(19-22):2393--2410,
  2006.

\bibitem{Joh1Kno2:2020-F}
V.~John, P.~Knobloch, and U.~Wilbrandt.
\newblock Finite element pressure stabilizations for incompressible flow
  problems.
\newblock {\em Adv. Math. Fluid Mech.}, 483--573, 2020.

\bibitem{Bur1Fer2:2008--G}
E.~Burman and M.~A. Fern\'{a}ndez.
\newblock Galerkin finite element methods with symmetric pressure stabilization
  for the transient {S}tokes equations: stability and convergence analysis.
\newblock {\em SIAM J. Numer. Anal.}, 47(1):409--439, 2008/09.

\bibitem{Bos1Vol2:2021--S}
G.-A.~Bosco, J.~Volker and N.~Julia. 
\newblock Symmetric pressure stabilization for equal-order ﬁnite element approximations to the time-dependent Navier--Stokes equations.
\newblock {\em IMA J. Numer. Anal. }, 41: 1093--1129, 2021.

\bibitem{Gan1Min2:2021--A}
C.~Gang and F.~Minfu,
\newblock Analysis of solving Galerkin ﬁnite element methods with symmetric pressure stabilization for the unsteady Navier--Stokes equations using conforming equal order interpolation.
\newblock {\em Adv. Appl. Math. Mech.}, 9: 362--377, 2017.

\bibitem{Coc1Kan2Sch3:2009--motified}
B.~Cockburn, G.~Kanschat, and D.~Sch\"{o}tzau.
\newblock An equal-order {DG} method for the incompressible {N}avier-{S}tokes
  equations.
\newblock {\em J. Sci. Comput.}, 40(1-3):188--210, 2009.

\bibitem{Are1Kar2:2022--E}
A.~Aretaki, E.~N. Karatzas, and G.~Katsouleas.
\newblock Equal higher order analysis of an unfitted discontinuous {G}alerkin
  method for {S}tokes flow systems.
\newblock {\em J. Sci. Comput.}, 91(2):48, 2022.




\bibitem{Hou1Han2:2021--motified}
Y.~Hou, Y.~Han, and J.~Wen.
\newblock An equal-order hybridized discontinuous {G}alerkin method with a
  small pressure penalty parameter for the {S}tokes equations.
\newblock {\em Comput. Math. Appl.}, 93:58--65, 2021.


\bibitem{Ces1Coc2--2013A}
A.~Cesmelioglu, B.~Cockburn, N.~C. Nguyen, and J.~Peraire.
\newblock Analysis of {HDG} methods for {O}seen equations.
\newblock {\em J. Sci. Comput.}, 55(2):392--431, 2013.

\bibitem{Gia1Sev2:2020--T}
M.~Giacomini, R.~Sevilla, and A.~Huerta.
\newblock Tutorial on hybridizable discontinuous {G}alerkin ({HDG}) formulation
  for incompressible flow problems.
\newblock In {\em Modeling in engineering using innovative numerical methods
  for solids and fluids}, volume 599 of {\em CISM Courses and Lect.},
  163--201. Springer, 2020.

\bibitem{Sol1Var2:2022--motified2}
M.~Solano and F.~Vargas~M.
\newblock An unfitted {HDG} method for {O}seen equations.
\newblock {\em J. Comput. Appl. Math.}, 399(18):113721, 2022.

\bibitem{Tu1Zha2:2023--B}
X.~Tu and J.~Zhang.
\newblock B{DDC} algorithms for {O}seen problems with {HDG} discretizations.
\newblock {\em IMA J. Numer. Anal.}, 43(6):3478--3521, 2023.

\bibitem{Cas1Seq2:2013--C}
P.~E. Castillo and F.~A. Sequeira.
\newblock Computational aspects of the local discontinuous {G}alerkin method on
  unstructured grids in three dimensions.
\newblock {\em Math. Comput. Modelling}, 57(9-10):2279--2288, 2013.

\bibitem{Gir1Rav2:2012--F}
V.~Girault and P.-A. Raviart.
\newblock {\em Finite Element Methods for Navier--Stokes Equations: Theory and
  Algorithms}, volume~5.
\newblock Springer Science \& Business Media, 2012.

\bibitem{Di1Dro2:2020--motified}
D.~A. Di~Pietro and J.~Droniou.
\newblock {\em The Hybrid High-order Method for Polytopal Meshes}, volume~19 of
  {\em Modeling, Simulation and Applications}.
\newblock Springer, Cham, 2020.

\bibitem{Di1Ern2:2012--M}
D.~A. Di~Pietro and A.~Ern.
\newblock {\em Mathematical Aspects of Discontinuous {G}alerkin Methods},
  volume~69 of {\em Math\'{e}matiques} \& {\em Applications (Berlin)}.
\newblock Springer, Heidelberg, 2012.

\bibitem{Cro18jTho2:1987--motified}
M.~Crouzeix and V.~Thom\'{e}e.
\newblock The stability in {$L_p$} and {$W^1_p$} of the {$L^2$}-projection onto
  finite element function spaces.
\newblock {\em Math. Comp.}, 48(178):521--532, 1987.

\bibitem{Che1:2021--O}
Y.~Cheng.
\newblock On the local discontinuous {G}alerkin method for singularly perturbed
  problem with two parameters.
\newblock {\em J. Comput. Appl. Math.}, 392:113485, 22, 2021.


\bibitem{Riv:2008--D}
B.~Riviere.
\newblock {\em Discontinuous Galerkin Methods for Solving Elliptic and Parabolic Equations}, volume~15 of Frontiers in {\em Applied Mathematics}, Society for Industrial and Applied Mathematics, Philadelphia, 2008.


\end{thebibliography}

\end{document}